\documentclass[11pt]{article}
\usepackage[margin=1in]{geometry}
\usepackage{amsmath,amssymb,amsthm,amsfonts}
\usepackage{enumerate}
\usepackage{hyperref}
\usepackage[all]{xy}

\newtheorem{thm}{Theorem}[section]
\newtheorem{theorem}[thm]{Theorem}

\newtheorem{proposition}[thm]{Proposition}
\newtheorem{cor}[thm]{Corollary}
\newtheorem{prop}[thm]{Proposition}
\newtheorem{lma}[thm]{Lemma}

\theoremstyle{definition}
\newtheorem{definition}[thm]{Definition}
\newtheorem{dfn}[thm]{Definition}
\newtheorem{notation}[thm]{Notation}
\newtheorem{rmk}[thm]{Remark}

\newtheorem{innercustomgeneric}{\customgenericname}
\providecommand{\customgenericname}{}
\newcommand{\newcustomtheorem}[2]{%
  \newenvironment{#1}[1]
  {%
   \renewcommand\customgenericname{#2}%
   \renewcommand\theinnercustomgeneric{##1}%
   \innercustomgeneric
  }
  {\endinnercustomgeneric}
}

\newcustomtheorem{customthm}{Theorem}
\newcustomtheorem{customprop}{Proposition}
\newcustomtheorem{customcor}{Corollary}

\newcommand{\bL}{\mathbf{L}}
\newcommand{\bG}{\mathbf{G}}
\newcommand{\oGpd}{\mathbf{oGpd}}
\newcommand{\lcCat}{\mathbf{lcCat}}

\newcommand{\id}{\mathrm{Id}}
\newcommand{\calA}{\mathcal{A}}
\newcommand{\calB}{\mathcal{B}}
\newcommand{\calC}{\mathcal{C}}
\newcommand{\calD}{\mathcal{D}}
\newcommand{\calG}{\mathcal{G}}
\newcommand{\calH}{\mathcal{H}}
\newcommand{\calI}{\mathcal{I}}
\newcommand{\calJ}{\mathcal{J}}
\newcommand{\calK}{\mathcal{K}}
\newcommand{\sh}{\mathrm{Sh}}

\begin{document}

\date{}
\title{A Double Categorical View on Representations of Etendues}
\author{
Darien DeWolf
\thanks{
  Department of Mathematics and Statistics,
  St. Francis Xavier University,
  2323 Notre Dame Ave,
  Antigonish, NS B2G 2W5,
  CANADA,
  ddewolf@stfx.ca}\,
and
Dorette Pronk
\thanks{
  Department of Mathematics and Statistics,
  Dalhousie University,
  6316 Coburg Road,
  Halifax, NS B3H 4R2,
  CANADA,
  Dorette.Pronk@Dal.Ca}}
\maketitle
\begin{abstract}
In this paper we introduce a description of ordered groupoids as a particular type of double categories. This enables us to turn Lawson's correspondence between ordered groupoids and left-cancellative categories into a biequivalence.
We use this to identify which ordered functors are maps of sites in the sense that they give rise to geometric morphisms between the induced sheaf categories,
and establish a Comparison Lemma for maps between Ehresmann sites.
\end{abstract}

\section{Introduction}
\label{sec:introduction}
As introduced in SGA4 \cite{sga4}, a topological \'etendue
\(\mathcal{E}\) is a topos which is locally a topological space:
there is some object $S\in\mathcal{E}$ together with a unique epimorphism
\(\xymatrix@1{S\ar@{->>}[r] & 1}\) such that \(\mathcal{E}/S\)
is equivalent to the topos of sheaves on a topological space.
By common convention (see, e.g., \cite{johnstone2002}),
we consider the more general localic \'etendues,
hereafter simply called \'etendues,
in which locales are used in lieu of topological spaces.

Rosenthal \cite{rosenthal1981} showed that the category of sheaves on a
left-cancellative site is an \'etendue
and, conversely,
Kock and Moerdijk \cite{kock-moerdijk-1991} showed that any \'etendue
is equivalent to the topos of sheaves on a left-cancellative site.
This presentation of general \'etendues has motivated the
subsequent work eventually leading to this paper.

Left-cancellative categories arise naturally in the study of cohomology
generalized from the context of groups to the context of inverse semigroups:
the cohomology of an inverse semigroup \cite{lausch1975}
is the same as the cohomology of a certain left-cancellative category
\cite{leech1987,lognathan1981}.
In particular,
the relationship between these cohomologies relies on a correspondence
between
certain actions of an inverse semigroup and
the actions on its associated left-cancellative category.
The study of inverse semigroups can also be done via ordered groupoids
as per the celebrated Ehresmann-Schein-Nambooripad Theorem
\cite{ehresmann:1960a,nambooripad2,nambooripad3,schein1}:
the category of inverse semigroups (and semigroup homomorphisms)
is equivalent to
the category of inductive groupoids (and inductive functors).
The Ehresmann-Schein-Nambooripad Theorem
has been nicely presented with its applications to
inverse semigroup theorem in Lawson's book \cite{lawson1998}.
and has since been extended to various natural contexts
\cite{cockett2019,dewolf-2018-a,gould2009,Hollings:2010iq,wang2019}.
Motivated by ordered groupoids being special types of inductive groupoids
and by the role of inverse semigroups acting on presheaves
in inverse semigroup theory \cite{meakin1998},
Lawson and Steinberg \cite{lawson2004}
engaged in this study of generalized group cohomology using inverse
semigroups in the more general context of ordered groupoids.

Lawson and Steinberg were successful in their investigation in that they
gave a first link between the topos-theoretic view
coming from sheaves on left-cancellative sites
(coming again from the relationship between cohomologies)
and the ordered-groupoid-theoretic view coming from the appropriate
sheaves on what they call Ehresmann sites;
Ehresmann sites are ordered groupoids equipped with
what they call an Ehresmann topology,
families of order ideals reminiscent of Grothendieck topologies.
They give a notion of sheaves on Ehresmann sites and prove:
\begin{enumerate}
\itemsep=0cm
\item
Each site with monic maps can be constructed from some Ehresmann site.
\item
Each \'etendue is equivalent to the category of sheaves on some Ehresmann site.
\end{enumerate}

To accomplish this,
Lawson and Steinberg define a pair of functors
\(\bL \colon \oGpd \rightarrow \lcCat\) and
\(\bG \colon \lcCat \rightarrow \oGpd\)
between the category of ordered groupoids (with ordered functors)
and the category of left-cancellative categories (with functors).
They then show that there is a natural transformation
\(\eta \colon \id \Rightarrow \bL\bG\)
with the property that for each left-cancellative category \(\calC\),
the component
\(\eta_\calC\colon \calC \rightarrow \bL\bG(\calC)\)
is a weak equivalence of categories.
Building off of this equivalence,
Lawson and Steinberg establish a one-to-one correspondence between
covering sieves of a left-cancellative site
\((\calC, J)\)
and the covering sieves in the corresponding left-cancellative site
\((\bL\bG(\calC), J_{T_J})\)
such that
the category of sheaves on \((\calC, J)\)
is equivalent to the category of sheaves on
\((\bL\bG(\calC), J_{T_J})\).

The purpose of Sections
\ref{sec:ordered-groupoids-as-double-categories} --
\ref{sec:the-equivalence-of-2-categories} of this paper
is primarily to strengthen Lawson and Steinberg's
result by answering the natural question
``Is there a corresponding natural transformation
\(\kappa\colon \bG\bL \Rightarrow \id \)
whose components are equivalences?''
Lawson and Steinberg provide a notion of such a natural transformation.
However, to recognize its components as weak equivalences,
one needs to view ordered groupoids as a kind of double category.
Double categories, as first introduced by Ehresmann \cite{Ehresmann1963},
have emerged as a convenient and powerful way to organize and study the
interaction between two different types of morphism on the same objects.
Given that partial orders can be thought of as categories,
we can re-define ordered groupoids as a special type of double category.
Our thinking of ordered groupoids as double categories allows us to tap
into the double-categorical toolbox;
in particular, thinking of ordered groupoids as double categories immediately
informs the correct notion of weak equivalence needed to answer our motivating
question in the positive:
once having defined ordered groupoids as double categories,
we are able to prove that the components of the natural transformation
\(\kappa\colon\id\Rightarrow\bG\bL\) given by Lawson and Steinberg
are weak equivalences in the sense of \cite{bunge1979}.

Pushing this further, we would like to say that we can establish an equivalence
of categories
\(\lcCat\simeq\oGpd\).
However, since the components of the natural transformations
\(\eta\) and \(\kappa\) are only (weak) equivalences, rather than isomorphisms,
we will need to consider \(\oGpd\) and \(\lcCat\) as 2-categories to do this.
We denote these 2-categories by
\textbf{\textit{oGpd}} and \textbf{\textit{lcCat}}.
The 2-structure of \textbf{\textit{lcCat}} is inherited from \(\mathbf{Cat}\):
the 2-cells are natural transformations.
To describe \textbf{\textit{oGpd}} as a 2-category requires more work in
choosing the correct notion of 2-cells.
We will call our choice of 2-cells \(\Lambda\)-transformations.
The existence of \(\Lambda\)-transformations depends on the fibration
(restriction) property of ordered groupoids giving the hom double category
\(\textbf{DblCat}(\mathcal{G}, \mathcal{H})\)
itself the structure of an ordered groupoid.
This way we obtain a 2-adjunction,

\begin{customthm}{\ref{2-adjunction}}
The 2-functors $\mathbf{L}\colon \mbox{\textbf{\textit{oGpd}}}\to\mbox{\textbf{\textit{lcCat}}}$ and
$\mathbf{G}\colon\mbox{\textbf{\textit{lcCat}}}\to\mbox{\textbf{\textit{oGpd}}}$ define a 2-adjunction,
$$\mbox{\textbf{\textit{oGpd}}}\simeq \mbox{\textbf{\textit{lcCat}}}.$$
\end{customthm}

To obtain a biequivalence the components of $\eta$ and $\kappa$ need to have weak inverses. In general this is not the case for $\kappa$.
However, we note that the ordered groupoids in the image of the functor $\mathbf{G}$ have the property that for each object $X$ there is a maximal object $\hat{X}$ such that  $X\le \hat{X}$. Restricting the 2-adjunction above to ordered groupoids with this property yields the desired biequivalence,

\begin{customcor}{\ref{2-equivalence}}
The 2-functors $\mathbf{L}\colon \mbox{\textbf{\textit{oGpd}}}_{\mbox{\scriptsize max}}\to\mbox{\textbf{\textit{lcCat}}}$ and
$\mathbf{G}\colon\mbox{\textbf{\textit{lcCat}}}\to\mbox{\textbf{\textit{oGpd}}}_{\mbox{\scriptsize max}}$ define a 2-adjoint biequivalence,
$$\mbox{\textbf{\textit{oGpd}}}_{\mbox{\scriptsize max}}\simeq \mbox{\textbf{\textit{lcCat}}}.$$
\end{customcor}

Section \ref{sec:applications} of this paper contains applications to the
study of sheaves on Ehresmann sites and further extends the work of Lawson and
Steinberg \cite{lawson2004} in two significant ways:

\begin{enumerate}
  \item
  Lawson and Steinberg show that there is an isomorphism of categories
  \begin{equation*}
  \mbox{\bf PreSh}(\calG)\cong\mbox{\bf PreSh}({\mathbf L}(\calG)).
  \end{equation*}
  They also show that this isomorphism restricts properly to sheaves
  with the chosen topologies.

  Furthermore, since any weak equivalence of categories induces an equivalence between the corresponding presheaf categories,
  we have
  $$
  \mbox{\bf PreSh}(\calC)\simeq\mbox{\bf PreSh}(\mathbf{LG}(\calC)),
  $$
  and by combining these equivalences we obtain,
  $$
  \mbox{\bf PreSh}({\mathbf G}(\calC))\simeq\mbox{\bf PreSh}(\calC)\mbox{ and }\mbox{\bf PreSh}(\calG)\simeq\mbox{\bf PreSh}(\mathbf{GL}(\calG)).
  $$
  We show that this equivalence also restricts properly to sheaves
  with the chosen topologies.
  Finally, while Lawson and Steinberg were able to establish an equivalence
  between categories of sheaves on the left-cancellative
  Grothendieck site side,
  our double-categorical perspective allows us to complete the picture and
  establish an equivalence between the categories of sheaves on the Ehresmann
  site side.
  \begin{customprop}{\ref{prop:sheaves-G-and-LG-and-C-and-GC}}
  \begin{enumerate}
    \item
    (Lawson and Steinberg)
    The category of sheaves on an Ehresmann site \((\calG, T)\)
    is equivalent to the category of sheaves on \((\bL(\calG), J_T)\).
    \item
    The category of sheaves on a left-cancellative site \((\calC, J)\)
    is equivalent to the category of sheaves on \((\bG(\calC), T_J)\).
  \end{enumerate}
  \end{customprop}
  \item
  We give an appropriate notion of morphism between Ehresmann sites which
  allows us to take the equivalences between categories of sheaves at the
  object level to an equivalence between the larger 2-categories of
  Grothendieck sites and of Ehresmann sites.

  This is motivated by Karazeris' \cite{Karazeris2004} result that functors
  between Grothendieck
  sites  give rise to geometric morphisms precisely when they are covering
  preserving and covering flat, and we prove the
  corresponding result for double functors between ordered groupoids:

  \begin{customthm}{\ref{geom}}
  If a functor
  \(M\colon(\mathcal{G}, T)\rightarrow (\mathcal{G}', T')\)
  of Ehresmann sites is covering preserving and covering flat,
  then \(M\) induces a geometric morphism
  \(\sh(M)\colon\sh(\mathcal{G}', T')\rightarrow \sh(\mathcal{G}, T)\).\qed
  \end{customthm}

  It is such functors that we call morphisms of Ehresmann sites
  which give a 2-category of Ehresmann sites that features in the following
  biequivalence.

  \begin{customthm}{\ref{adj-sites}}
  The functors $\mathbf{G}$ and $\mathbf{L}$ induce a 2-adjoint biequivalence
  $$
  \mbox{\bf lcGsite}\simeq\mbox{\bf Esite}_{\mbox{\scriptsize max}}.
  $$
  \end{customthm}
\end{enumerate}

The Comparison Lemma in \cite{kock-moerdijk-1991} gives sufficient conditions
on a morphism of sites so that it induces an equivalence between the
corresponding categories of sheaves.
As a final application, this paper
adapts Kock and Moerdijk's conditions to the context of ordered groupoids,
we are able to express and prove an analogous result for a morphism of
Ehresmann sites:

\begin{customthm}{\ref{thm:cl-ehresmann}}
[Comparison Lemma for Ehresmann Sites]
Let $M: (\mathcal{G}, T)\rightarrow (\mathcal{G}', T')$ be a morphism of
Ehresmann sites.
If $M$ is locally full, locally faithful, and locally surjective,
then the functor
$M^*: \sh(\mathcal{G}', T')\rightarrow \sh(\mathcal{G}, T)$
is full and faithful.
If further $M$ is co-continuous, then $M^*$ is an equivalence. \qed
\end{customthm}

\section{Ordered Groupoids as Double Categories}
\label{sec:ordered-groupoids-as-double-categories}
In order to describe the correspondence between ordered groupoids and left-cancellative categories in more detail,
we first introduce a new way of representing  ordered groupoids in terms of double categories.

\begin{dfn}
 An {\em ordered groupoid} is a category $\calG$  in which all arrows are invertible and such that
\begin{enumerate}
\item
There is a partial order relation on the arrows which extends to the objects via the identity arrows;
\item
The order is preserved by taking inverses and composition: if $a\le b$ then $a^{-1}\le b^{-1}$
and if $a\le b$ and $c\le d$ then $ac\le bd$;
\item
When $f\colon A\to B$ and $A'\le A$ there is a unique arrow $f'\colon A'\to B'$ such that $f'\le f$.
We also write $f|_{A'}$ for $f'$.
\end{enumerate}
\end{dfn}

Note that the first and second conditions in this definition imply that if $f\le g$ and $f\colon A\to B$
and $g\colon C\to D$ then $A\le C$ and $B\le D$. Hence, we can also view this as an internal groupoid
$$
\xymatrix{
\calG_1\times_{\calG_0}\calG_1\ar[r]^-m&\calG_1\ar[r]^i & \calG_1\ar@<.8ex>[r]^t\ar@<-.8ex>[r]_s
& \calG_0\ar[l]|u
}
$$
in the category of partially ordered sets with an additional property corresponding to the last requirement given above:
the domain arrow $\xymatrix@1{\calG_1\ar[r]^s&\calG_0}$ is a fibration as functor between posetal categories.
It follows from the groupoid symmetry that the target arrow $t$ is an opfibration.
So we observe that ordered groupoids have both domain and range restriction.

Another way to view this last diagram is as a double category $\calG$ where the vertical arrows give
the poset structure and the horizontal arrows give the groupoid structure.
Double cells have the following form
\begin{equation}\label{double_cell}
\xymatrix@R=1.8em{
X\ar[r]^g\ar[d]|-{\scriptscriptstyle\bullet}\ar@{}[dr]|\le &Y\ar[d]|-{\scriptscriptstyle\bullet}
\\
X'\ar[r]_{g'}&Y'}
\end{equation}
And this encodes that $X\le X'$, $Y\le Y'$ and $g\le g'$.
Note that in this notation, the fact that $s\colon \calG_1\to\calG_0$ is a fibration corresponds to the statement that
for each diagram
$$
\xymatrix@R=1.8em{
X\ar[d]|-{\scriptscriptstyle\bullet}
\\
X'\ar[r]_{g'}&Y'}
$$
there is a unique diagram (\ref{double_cell}).

The morphisms between ordered groupoids are usually taken to be ordered functors: functors that preserve the order relation.
These correspond precisely to
double functors between the double categories just described. We write {\bf oGpd} for the category of ordered groupoids, considered as double categories with double functors as arrows.

\section{Lawson's Correspondence Revisited}
\label{sec:lawsons-correspondence-revisited}
In \cite{L1} Lawson introduced a correspondence between ordered groupoids and left-cancellative categories; i.e., categories in which all arrows are monomorphisms. We write {\bf lcCat} for the category of left-cancellative categories with functors as morphisms.

Lawson introduced functors $\mbox{\bf oGpd} \to\mbox{\bf lcCat}$ and $\mbox{\bf lcCat} \to\mbox{\bf oGpd}$.
We begin by rewriting these functors in our terminology.

\subsection{The Functors \texorpdfstring{$\bL$}{Lg} and \texorpdfstring{$\bG$}{Lg}}
The functor  ${\mathbf L}\colon \mbox{\bf oGpd} \to\mbox{\bf lcCat}$ is defined as follows.
For an ordered groupoid $\calG$,  the left-cancellative category ${\mathbf L}({\calG})$ has as objects those of $\calG$.
An arrow $A\to B$ in  ${\mathbf L}({\calG})$ is a formal composite of a horizontal arrow in $\calG$ with a vertical arrow in $\calG$:
$$\xymatrix{A\ar[r]^-h & B'\ar[r]|-{\scriptscriptstyle\bullet}& B}$$
where $h$ is a horizontal arrow in $\calG$ and $\xymatrix@1{B'\ar[r]|-{\scriptscriptstyle\bullet}& B}$ is a vertical arrow in $\calG$.
Composition uses the restriction operation in $\calG$,
$$
\xymatrix@R=1.8em{
A\ar[r]^h &B'\ar[d]|-{\scriptscriptstyle\bullet} \ar[r]^{k|_{B'}}\ar@{}[dr]|\le & C''\ar[d]|-{\scriptscriptstyle\bullet}
\\
&B\ar[r]^k &C'\ar[d]|-{\scriptscriptstyle\bullet}
\\
&& C
}
$$
so the composition is given by $\xymatrix@1@C=3em{A\ar[r]^{k|_{B'}h}&C''\ar[r]|-{\scriptscriptstyle\bullet}&C}$. (Note that this is unitary and associative by the uniqueness of the restrictions.)

Conversely, the functor ${\mathbf G}\colon \mbox{\bf lcCat}\to \mbox{\bf oGpd}$ is defined as follows.
For a left-cancellative category $\mathcal C$, the ordered groupoid ${\mathbf G}({\mathcal C})$ has subobjects in $\mathcal C$
as objects; i.e., they are equivalence classes of arrows
$m\colon A\to B$
and $[m\colon A\to B]=[m'\colon A'\to B]$ if there is an isomorphism $k\colon A\stackrel{\sim}{\to} A'$ in $\calC$
such that
$$
\xymatrix{A\ar[rr]^k_\sim\ar[dr]_m && A'\ar[dl]^{m'}
\\
&B
}
$$
commutes.
The horizontal arrows in ${\mathbf G}({\mathcal C})$ are equivalence classes of spans,
$$[m,n]\colon [m]\to [n]$$
The equivalence relation is defined so that $[m,n]=[m',n']$ if and only if there is an isomorphism $h$ making the following diagram commute:
$$
\xymatrix@R=1em{
&A\ar[dl]_m\ar[dr]^n\ar[dd]^h_\wr
\\
B && C
\\
&A'\ar[ul]^{m'}\ar[ur]_{n'}
}
$$
Composition of $[k,m]$ and $[m',n]$ is defined when $[m]=[m']$; i.e., when there is an isomorphism $h$ such that $m'h=m$, giving rise to a diagram
$$
\xymatrix{
&\ar[dl]_k\ar[dr]_{m}\ar[rr]^h_\sim && \ar[dl]^{m'}\ar[dr]^n
\\
&&&&
}$$
in ${\mathcal C}$. The composition is then $\xymatrix@1@C=3em{[k]\ar[r]^-{[k,nh]}&[nh]=[n]}$.

The vertical arrows are given by the order relation on subobjects: there is a unique vertical arrow
$$\xymatrix{[n]\ar[r]|-{\scriptscriptstyle\bullet}&[n']}$$
if there is an arrow $h$ in $\mathcal C$ such that $n=n'h$; i.e., $[n]\le[n']$ as subobjects.

The order relation on arrows is defined by Lawson as: $[m,n]\le [m',n']$ if there is an arrow $h$ in ${\mathcal C}$ such that the diagram
\begin{equation}\label{le-diagrams}
\xymatrix@R=.8em@C=3em{
&A\ar[dl]_m\ar[dr]^n\ar[dd]^h
\\
B && C
\\
&A'\ar[ul]^{m'}\ar[ur]_{n'}
}
\end{equation}
commutes. (Note that this $h$ is unique if it exists.)
This implies then that $[m]\le[m']$ and $[n]\le[n']$.
So double cells in ${\mathbf G}(\mathcal C)$,
$$
\xymatrix@R=1.8em@C=3em{
[m]\ar[d]|-{\scriptscriptstyle\bullet}\ar[r]^{[m,n]}\ar@{}[dr]|\le & [m]\ar[d]|-{\scriptscriptstyle\bullet}
\\
[m']\ar[r]_{[m',n']}&[n']}
$$
correspond to diagrams of the form (\ref{le-diagrams}) in ${\mathcal C}$.
Since there is at most one double cell for any frame of horizontal and vertical arrows, the horizontal and vertical composition of double cells is
determined by the composition of the horizontal and vertical arrows.

\subsection{The Composition \texorpdfstring{$\bL\bG$}{Lg}}
\label{D:eta}
We now describe the results of composing the functors $\mathbf L$ and $\mathbf G$ in our terminology.
For a left-cancellative category $\mathcal C$, the category $\mathbf{LG}(\mathcal C)$ has as objects subobjects  in $\mathcal C$:
$[m\colon A\to B]$.

The arrows in $\mathbf{LG}(\mathcal C)$ are constructed as
$$\xymatrix{[m]\ar[r]^{[m,n']}&[n']\ar[r]|-{\scriptscriptstyle\bullet}&[n]}$$
and this corresponds to a diagram
$$
\xymatrix@R=1.5em@C=3em{
&A'\ar[dl]_m\ar[dr]_{n'}\ar[r]^{h} & A\ar[d]^{n}
\\
B &&C}
$$
in ${\mathcal C}$.

So an arrow $[h]\colon[m\colon A'\to B]\to [n\colon A\to C]$ is represented by an arrow $h\colon A'\to A$.
Furthermore,
$$
\left(\xymatrix@1{[m]\ar[r]^{[h]}&[n]}\right)\equiv\left(\xymatrix@1{[m']\ar[r]^{[h']}&[n']}\right)
$$
if and only if there are isomorphisms $k$ and $\ell$ that make the following diagram commute,
$$
\xymatrix@C=3em@R=1em{
&\ar[dl]_m\ar[r]^h\ar[dd]_\wr^k & \ar[dd]_\ell^\wr\ar[dr]^n
\\
&&&
\\
&\ar[ul]^{m'}\ar[r]_{h'} & \ar[ur]_{n'}
}
$$

Composition of $\xymatrix{[m]\ar[r]^{[h_1]}&[n]}$ and $\xymatrix{[n']\ar[r]^{[h_2]}&[p]}$
is defined when there is an arrow $k$ as in the diagram
$$
\xymatrix{
&\ar[dl]_m\ar[r]^{h_1}& \ar[dr]_{n}\ar[rr]^k & &\ar[dl]^{n'}\ar[r]^{h_2} & \ar[dr]^p
\\
&&&&&&
}
$$
and the composition is $$[h_2kh_1]\colon [m]\to[p].$$

The categories $\mathcal C$ and $\mathbf{LG}({\mathcal C})$ are not isomorphic, but as observed by Lawson \cite[Theorem 2.3.1]{L1}, there is a functor
$$\eta_{\mathcal C}\colon {\mathcal C}\to \mathbf{LG}({\mathcal C})$$
giving an equivalence of categories. (It is
defined on objects by $A\mapsto [1_A]$, and
on arrows by $\left(h\colon A\to B\right)\mapsto(\left[h]\colon [1_A]\to[1_B]\right)$ and note that $[m\colon A\to B]\cong [1_A\colon A\to A]$ and $\eta_{\calC}$ defines an isomorphism $\calC(A,B)\stackrel{\sim}{\longrightarrow}\mathbf{LG}(\calC)([1_A],[1_B])$.) Note that the $\eta_{\mathcal C}$ define a natural transformation $$\eta\colon 1_{\mbox{\scriptsize\bf lcCat}}\Rightarrow \mathbf{LG}.$$

\subsection{The Composition \texorpdfstring{$\bG\bL$}{Lg}}
\label{D:kappa}
For the other composition, $\mathbf{GL}\colon \mbox{\bf oGpd}\to\mbox{\bf oGpd}$, let $\mathcal G$ be an ordered groupoid. Then the objects of
$\mathbf{GL}({\mathcal G})$ are subobjects in $\mathbf{L}(\mathcal G)$, hence equivalence classes,
$$
\xymatrix@C=1.9em{[A\ar[r]^h& B'\ar[r]|-{\scriptscriptstyle\bullet}& B],}
$$
where $h$ is a horizontal arrow in $\mathcal G$ (and therefore invertible). Furthermore,
$$\xymatrix@1{[A\ar[r]^h& B'\ar[r]|-{\scriptscriptstyle\bullet}& B]}=\xymatrix@1{[A'\ar[r]^{h'}& B'\ar[r]|-{\scriptscriptstyle\bullet}& B]}$$ if and only if there is an isomorphism $k\colon A\stackrel{\sim}{\longrightarrow} A'$ such that $h'k=h$.
Note that in this case each equivalence class has a canonical representative, $\xymatrix{(B'\ar[r]|-{\scriptscriptstyle\bullet}& B)}$. We will denote this object by
$$(B', B).$$

Horizontal arrows in $\mathbf{GL}(\calG)$ become then equivalence classes of spans of horizontal arrows in $\calG$,
$$\left[\xymatrix{B'&A\ar[l]_-h\ar[r]^-k&C'}\right]\colon(B', B)\to(C',C)$$
Since $h$ and $k$ are invertible, this span is equivalent to
$$\xymatrix{B' &B'\ar[l]_{1_B}\ar[r]^{kh^{-1}} & C'.}$$
So a horizontal arrow $(B', B)\to (C', C)$ is given by a horizontal arrow $h\colon B'\to C'$ in $\mathcal G$. The vertical arrows and the double cells in $\mathbf{GL}(\mathcal G)$ are obtained as follows: there is a (unique) vertical arrow
$\xymatrix@1{(B',B)\ar[r]|-{\scriptscriptstyle\bullet}&(D', D)}$
if and only if $B=D$ and there is a vertical arrow $\xymatrix@1{B'\ar[r]|-{\scriptscriptstyle\bullet}&D'}$ in $\mathcal G$.
Similarly, double cells in $\mathbf{GL}(\calG)$ are of the form,
$$
\xymatrix@R=1.8em{
(B',  D)\ar[d]|-{\scriptscriptstyle\bullet}\ar@{}[dr]|\le\ar[r]^h & (C', E)\ar[d]|-{\scriptscriptstyle\bullet}
\\
(D',D)\ar[r]_k & (E', E)
}
$$
where
$$
\xymatrix@R=1.8em{
B'\ar[d]|-{\scriptscriptstyle\bullet}\ar@{}[dr]|\le\ar[r]^h & C'\ar[d]|-{\scriptscriptstyle\bullet}
\\
D'\ar[r]_k & E'
}
$$
is a double cell in $\mathcal G$.

Lawson introduced an ordered functor $\kappa_\calG\colon\mathbf{GL}(\calG)\to\calG$ which corresponds to the following double functor with the same name:
\begin{itemize}
\item
on objects,
$\kappa_\calG(B', B)=B'$;
\item
on horizontal arrows, $\kappa_\calG((B', B)\stackrel{h}{\to}(C', C))=(B'\stackrel{h}{\to}C')$;
\item
on vertical arrows, $\kappa_{\calG}(\xymatrix@1{(B', B)\ar[r]|-{\scriptscriptstyle\bullet}&(D', B)})=(\xymatrix@1{B'\ar[r]|-{\scriptscriptstyle\bullet} & D'})$;

\item
on double cells, $\kappa_{\mathcal G}$ maps the cell
$$
\xymatrix{
(B', B)\ar[d]|-{\scriptscriptstyle\bullet}\ar@{}[dr]|\le\ar[r]^h & (C', C)\ar[d]|-{\scriptscriptstyle\bullet}
\\
(D', D)\ar[r]_k & (E', E)
}
$$
in $\mathbf{GL}(\calG)$
to the cell
$$
\xymatrix@R=1.8em{
B'\ar[d]|-{\scriptscriptstyle\bullet}\ar@{}[dr]|\le\ar[r]^h & C'\ar[d]|-{\scriptscriptstyle\bullet}
\\
D'\ar[r]_k & E'
}
$$
in $\calG$.
\end{itemize}

Lawson was not able to show that this is a weak equivalence, because it is not clear a priori what a weak equivalence of ordered groupoids should be. However, in the language of double categories this problem has been resolved in the literature: \cite{bunge1979} gives a description of internal weak equivalences in terms of effective descent maps, and \cite{EKV} shows that these weak equivalences are part of a Quillen model structure on the category of double categories, induced by the regular epimorphism topology on the category of categories.

\begin{dfn}
A functor between internal categories $F\colon{\mathbb C}\to{\mathbb D}$ in some ambient category $\mathcal D$ is a {\em weak equivalence} if
it satisfies the following two conditions:
\begin{enumerate}
\item
It is {\em essentially surjective} in the sense that the composition of the top arrows in
$$
\xymatrix{
{\mathbb C}_0\times_{{\mathbb D}_0}{\mathbb D}_1\ar[r]^{\pi_2}\ar[d]_{\pi_1} & {\mathbb D}_1\ar[d]_s\ar[r]^t &{\mathbb D}_0
\\
{\mathbb C}_0\ar[r]_{F_0}&{\mathbb D}_0}
$$
is of effective descent in $\mathcal D$;
\item
It is {\em fully faithful} in the sense that the following square is a pullback,
$$
\xymatrix@C=4em{
{\mathbb C}_1\ar[r]^{F_1}\ar[d]_{(s,t)} & {\mathbb D}_1\ar[d]^{(s,t)}
\\
{\mathbb C}_0\times{\mathbb C}_0\ar[r]_{F_0\times F_0}&{\mathbb D}_0\times{\mathbb D}_0
}
$$
\end{enumerate}
\end{dfn}

For $\mathcal D=\mbox{\bf Cat}$, the category of small categories, internal categories are double categories and it was shown in \cite{JST}
that a functor $F\colon {\mathcal X}\to{\mathcal Y}$ is of effective descent if and only if the following induced functions of sets are surjective: $F_0\colon {\mathcal X}_0\to{\mathcal Y}_0$, $F_1\colon {\mathcal X}_1\to{\mathcal Y}_1$ and
$F_1\times F_1\colon {\mathcal X}_1\times_{{\mathcal X}_0}{\mathcal X}_1\to{\mathcal Y}_1\times_{{\mathcal Y}_0}{\mathcal Y}_1$.

\begin{prop}
The double functor $\kappa_\calG\colon\mathbf{GL}(\calG)\to\calG$
is a weak equivalence of double categories.
\end{prop}

\begin{proof}
We first check that $\kappa_\calG$ is essentially surjective on objects.
So we need to check that the induced functor $t\pi_2\colon \mathbf{GL}(\calG)_0\times_{\calG_0}\calG_1\to\calG_0$ is of effective descent.
\begin{itemize}
\item
It is surjective on objects, because for any object $B$ in $\calG$, $\kappa_\calG(B, B)=B$.
\item
It is surjective on arrows, because the arrows in $\calG_0$ are the vertical arrows of $\calG$, and for any vertical arrow $\xymatrix@1{B\ar[r]|-{\scriptscriptstyle\bullet}&C}$,
$\kappa_\calG\xymatrix@1{((B, C)\ar[r]|-{\scriptscriptstyle\bullet}&(C, C))}=\xymatrix@1{(B\ar[r]|-{\scriptscriptstyle\bullet}&C).}$
\item
Finally, $\kappa_{\calG}\times\kappa_{\calG}\colon\mathbf{GL}(\calG)_1\times_{\mathbf{GL}(\calG)_0}\mathbf{GL}(\calG)_1\to \calG_1\times_{\calG_0}\calG_1$ is surjective, since  for any composable pair of vertical arrows in $\calG$, $\xymatrix@1{B\ar[r]|-{\scriptscriptstyle\bullet}&C \ar[r]|-{\scriptscriptstyle\bullet}& D}$,
$$(\kappa_\calG\times\kappa_\calG)\left(\xymatrix@1@C=2em{(B, D)\ar[r]|-{\scriptscriptstyle\bullet}&(C, D)\ar[r]|-{\scriptscriptstyle\bullet}&(D, D)}\right)=\left(\xymatrix@1@C=2em{B\ar[r]|-{\scriptscriptstyle\bullet}&C\ar[r]|-{\scriptscriptstyle\bullet}&D}\right).$$
\end{itemize}
We note that $\kappa_{\calG}$ is fully faithful because it is both order reflecting and order preserving.\end{proof}

We would like to combine the results from this section
in saying that the functors $\mathbf{L}\colon \mbox{\bf oGpd}\to\mbox{\bf lcCat}$ and $\mathbf{G}\colon\mbox{\bf lcCat}\to\mbox{\bf oGpd}$ define an equivalence of categories $\mbox{\bf lcCat}\simeq\mbox{\bf oGpd}$.
However, since the components of the natural transformations $\eta\colon 1_{\mbox{\bf\scriptsize lcCat}}\Rightarrow \mathbf {LG}$ and $\kappa\colon \mathbf{GL}\Rightarrow 1_{\mbox{\bf\scriptsize oGpd}}$ are only (weak) equivalences, rather than isomorphisms, we will need to consider {\bf oGpd} and {\bf lcCat} as 2-categories to do this. We will denote these 2-categories by \textbf{\textit{oGpd}} and \textbf{\textit{lcCat}}. The 2-structure of \textbf{\textit{lcCat}} is inherited from {\bf Cat}: the 2-cells are natural transformation. To describe \textbf{\textit{oGpd}} as a 2-category we need to do more work as spelled out in the next section.

\section{\textbf{\textit{oGpd}} as a 2-Category}
\label{sec:ogpd-as-a-2-category}
We clearly want the arrows of the 2-category \textbf{\textit{oGpd}} to be double functors.
When constructing a 2-category from a double category one chooses usually either the horizontal or the vertical transformations as the 2-cells of the resulting 2-category.
The components of a horizontal transformation are horizontal arrows and double cells in the codomain double category, so for ordered groupoids, all horizontal transformations are invertible. The components of a vertical transformation are vertical arrows and double cells, so there is a vertical transformation
\(F\Rightarrow_v G\colon\calG\rightarrow\calH\)
if and only if \(F\leq G\).

However, because each ordered groupoid has a fibration as domain, we obtain a third option. To describe this third option, first recall that for any two double categories $\mathbb C$ and $\mathbb D$, $\mbox{\bf DblCat}({\mathbb C},{\mathbb D})$
can be viewed as a double category with double functors as objects, horizontal transformations as horizontal arrows, vertical transformations as vertical arrows and modifications as double cells.

\begin{prop}
For ordered groupoids $\calG$ and $\calH$, the double category $$\mbox{\bf DblCat}(\calG,\calH)$$ is again an ordered groupoid.
\end{prop}

\begin{proof}
We saw above that the vertical transformations simply encode the order structure on the double functors and all horizontal transformations are invertible.

We now describe what the modifications are in this double category.
For four double functors $F, G, H, K\colon \calG\to\calH$ with $F\le H$ and $G\le G$ and horizontal transformations $\alpha\colon F\stackrel{\sim}{\longrightarrow}G$
and $\beta\colon H\stackrel{\sim}{\longrightarrow}K$, a modification $\Theta$,
$$
\xymatrix@R=1.8em{
F\ar[r]^\alpha\ar[d]|-{\scriptscriptstyle\bullet}\ar@{}[dr]|\Theta & G\ar[d]|-{\scriptscriptstyle\bullet}
\\
H\ar[r]_\beta& K
}
$$
is given by a family of double cells
$$
\xymatrix@R=1.8em{
FX\ar[r]^{\alpha_X}\ar[d]|-{\scriptscriptstyle\bullet}\ar@{}[dr]|{\Theta_X} & GX\ar[d]|-{\scriptscriptstyle\bullet}
\\
HX\ar[r]_{\beta_X}& KX
}
$$
in $\calH$, indexed by objects $X$ in $\calG$, and satisfying certain naturality conditions.
However, $\calH$ has only double cells of the form
$$\xymatrix@R=1.8em{\ar[r]^\sim\ar[d]|-{\scriptscriptstyle\bullet}\ar@{}[dr]|\le&\ar[d]|-{\scriptscriptstyle\bullet}\\
\ar[r]_\sim &
}
$$
Hence, each cell $\Theta_X$ is the unique double cell encoding the fact that $\alpha_X\le \beta_X$.
So we may write
$$
\xymatrix@R=1.8em{
F\ar[r]^\alpha\ar[d]|-{\scriptscriptstyle\bullet}\ar@{}[dr]|\le & G\ar[d]|-{\scriptscriptstyle\bullet}
\\
H\ar[r]_\beta& K
}
$$
for $\Theta$.

It remains to show that the domain arrow $$\xymatrix{s\colon\mbox{\bf DblCat}(\calG,\calH)_1\ar[r]&\mbox{\bf DblCat}(\calG,\calH)_0}$$ is a fibration.
So suppose that $F\le H$ and $\beta\colon H\Rightarrow_h K$.
We construct $\beta|_r$ as follows.
For each object $X$ in $\calG$, we have
$$
\xymatrix@R=1.8em{
FX\ar[d]|-{\scriptscriptstyle\bullet}\\
HX\ar[r]_{\beta_X}& KX}
$$
and we use the lifting property of $\calH$ to complete the square
$$
\xymatrix@R=1.8em@C=4em{
FX\ar[d]|-{\scriptscriptstyle\bullet}\ar@{}[dr]|\le\ar[r]^{\beta_X|_{FX}}& GX\ar[d]|-{\scriptscriptstyle\bullet}\\
HX\ar[r]_{\beta_X}& KX}
$$
To turn this assignment of $G$ into a double functor,
consider a horizontal arrow $h\colon X\to Y$ in $\calG$.
We have by horizontal naturality that $Kh\circ \beta_X=\beta_y\circ Hh$,
and we have the following double cells,
$$
\xymatrix@R=1.8em@C=4em{
FX\ar[d]|-{\scriptscriptstyle\bullet}\ar@{}[dr]|{\le}\ar[r]^{Fh}
  &FY \ar[d]|-{\scriptscriptstyle\bullet}\ar@{}[dr]|{\le}\ar[r]^{\beta_Y|_{FY}}
  &GY \ar[d]|-{\scriptscriptstyle\bullet}
\\
HX\ar[r]_{Hh}& HY\ar[r]_{\beta_Y}&KY
}
$$
So we see that in
$$
\xymatrix@R=1.8em@C=4em{
FX\ar[d]|-{\scriptscriptstyle\bullet}\ar@{}[dr]|\le\ar[r]^{\beta_X|_{FX}} &GX\ar[d]|-{\scriptscriptstyle\bullet}\\
HX\ar[r]_{\beta_X}&KX\ar[r]_{Kh}& KY
}
$$
the codomain of the lifting $Kh|_{GX}$ has codomain $FY$ (since
$(Kh|_{GX})\circ(\beta_X|_{FX})=(Kh\circ\beta_X)|_{FX}=(\beta_Y\circ Hh)|_{FX}=(\beta_Y|_{FY})\circ Fh$). So we may define $Gh=Kh|_{GX}$.
Thus defined, $G$ preserves identities and composition, because the liftings are unique. We also see from the diagrams above that $Gh\circ(\beta_X|_{FX})=(\beta_Y|_{FY})\circ Fh$.

For the definition of $G$ on vertical arrows, suppose that $X\le X'$.
Then we have the following composites of vertical arrows in $\calH$:
$$
\xymatrix{
FX\ar[r]|-{\scriptscriptstyle\bullet} &HX\ar[r]|-{\scriptscriptstyle\bullet}& HX'}\mbox{ and }\xymatrix{FX\ar[r]|-{\scriptscriptstyle\bullet} &FX'\ar[r]|-{\scriptscriptstyle\bullet}& HX'}$$
Hence the horizontal arrow $\xymatrix@1{HX'\ar[r]^{\beta_{X'}}&KX'}$ can be
restricted to $FX$ in two ways:
$$
\xymatrix@R=1.8em@C=4em{
FX\ar[d]|-{\scriptscriptstyle\bullet}\ar@{}[dr]|\le\ar[r]^{\beta_{X'}|_{FX}} & G'X\ar[d]|-{\scriptscriptstyle\bullet} && FX\ar[d]|-{\scriptscriptstyle\bullet}\ar@{}[dr]|\le\ar[r]^{\beta_{X}|_{FX}} & GX\ar[d]|-{\scriptscriptstyle\bullet}
\\
FX'\ar[d]|-{\scriptscriptstyle\bullet}\ar@{}[dr]|\le\ar[r]_{\beta_{X'}|_{FX'}}&GX'\ar[d]|-{\scriptscriptstyle\bullet} &\mbox{and}& HX\ar[d]|-{\scriptscriptstyle\bullet}\ar@{}[dr]|\le\ar[r]_{\beta_{X}}&KX\ar[d]|-{\scriptscriptstyle\bullet}
\\
GX'\ar[r]_{\beta_{X'}}&KX'&&HX'\ar[r]_{\beta_{X'}}&KX'
}
$$
Hence, $G'X=GX$ and $\xymatrix@1{GX\ar[r]|-{\scriptscriptstyle\bullet}&GX'}$ as required.

Finally, to define $G$ on double cells, let
$$
\xymatrix@R=1.8em{
X\ar[d]|-{\scriptscriptstyle\bullet}\ar[r]^h\ar@{}[dr]|\le&Y\ar[d]|-{\scriptscriptstyle\bullet}
\\
X'\ar[r]_{h'}&Y'
}
$$
be a double cell in $\calG$.
We calculate the restriction $Kh'|_{GX}$ in two different ways.
First we take the following factorization,
$$
\xymatrix@R=1.8em{
GX\ar[d]|-{\scriptscriptstyle\bullet}\ar[r]^{Gh}\ar@{}[dr]|\le &GY\ar[d]|-{\scriptscriptstyle\bullet}
\\
KX\ar[d]|-{\scriptscriptstyle\bullet}\ar@{}[dr]|\le \ar[r]_{Kh}&KY\ar[d]|-{\scriptscriptstyle\bullet}
\\
KX'\ar[r]_{Kh'}&KY'}
$$
This shows that $Kh'|_{FX}=Kh|_{FX}=Gh$.
Now consider
$$
\xymatrix@R=1.8em@C=4em{
GX\ar[d]|-{\scriptscriptstyle\bullet}\ar[r]^{Gh'|_{GX}}\ar@{}[dr]|\le &GY\ar[d]|-{\scriptscriptstyle\bullet}
\\
GX'\ar[d]|-{\scriptscriptstyle\bullet}\ar@{}[dr]|\le \ar[r]_{Gh'}&GY'\ar[d]|-{\scriptscriptstyle\bullet}
\\
KX'\ar[r]_{Kh'}&KY'}
$$
This shows that $Gh=Kh'|_{GX}=Gh'|_{GX}$
and hence we have the double cell
$$
\xymatrix@R=1.8em{
GX\ar[d]|-{\scriptscriptstyle\bullet}\ar@{}[dr]|\le\ar[r]^{Gh}&GY\ar[d]|-{\scriptscriptstyle\bullet}
\\
GX'\ar[r]_{Gh'}&GY'}
$$
as required.
\end{proof}

In summary, we can apply our functor $\mathbf{L}$ to the ordered groupoid $\mbox{\bf DblCat}(\calG,\calH)$
to obtain a left-cancellative category $\mathbf{L}\left(\mbox{\bf DblCat}(\calG,\calH)\right)$.
This allows us to define the 2-category \textbf{\textit{oGpd}} with ordered groupoids as objects and $\mbox{\textbf{\textit{oGpd}}}(\calG,\calH)=\mathbf{L}\left(\mbox{\bf DblCat}(\calG,\calH)\right)$. This means that
a 2-cell $$(\alpha,\le)\colon F\Rightarrow G$$ is a formal composite
$$\xymatrix{F\ar@{=>}[r]^-\sim_-\alpha &G'\le G},$$
where $G'\colon\calG\to\calH$ is a double functor, $\alpha$ is a horizontal transformation  and $\le$ denotes a vertical transformation as described above.
We call such a formal composite a $\Lambda$-transformation.
Vertical composition of these $\Lambda$-transformations is given by composition in $\mathbf{L}\left(\mbox{\bf DblCat}(\calG,\calH)\right)$, using the fibration property of the domain map.

To define horizontal composition note that since both horizontal and vertical transformations allow for left and right whiskering, whiskering automatically extends to $\Lambda$-transformations:
Given ordered groupoids $\calG$, $\calH$ and $\calK$ with double functors
$$
\xymatrix{\calG\ar@<.5ex>[r]^F\ar@<-.5ex>[r]_{G}&\calH\ar@<.5ex>[r]^H\ar@<-.5ex>[r]_K&\calK}
$$
and $\Lambda$-transformations $(\alpha,\le)\colon F\Rightarrow G$ and $(\beta,\le)\colon H\Rightarrow K$, we have that
$$(\beta,\le)F=(\beta F,\le)\colon HF\Rightarrow KF$$
and
$$H(\alpha,\le)=(H\alpha,\le)\colon HF\Rightarrow HG.$$
We want to show  that this gives rise to a well-defined notion of horizontal composition.
In the proof we will need the following results about horizontal transformations between double functors of ordered groupoids.

\begin{lma}\label{middle-four-tricks}
Let $\calG$, $\calH$ and $\calK$ be ordered groupoids with double functors $$
\xymatrix{\calG\ar@<1ex>[r]^F\ar[r]|{G'}\ar@<-1ex>[r]_{G}&\calH\ar@<1ex>[r]^H\ar[r]|{K'}\ar@<-1ex>[r]_{K}&\calK}
$$ with horizontal transformations $\alpha\colon F\Rightarrow G'$ and $\beta\colon H\Rightarrow K'$ and vertical transformations $G'\le G$ and $K'\le K$.
Then we have the following restrictions in $\mbox{\bf DblCat}(\calG,\calK)$:
\begin{enumerate}
\item
$(\beta G)|_{HG'}=\beta G'$.
\item
$(K\alpha)|_{K'F}=K'\alpha$.
\end{enumerate}
\end{lma}

\begin{proof}
Since the restrictions are unique, we need only to check that the assigned horizontal transformations fit.
So let $X$ be an object in $\calG$.

For the first restriction, we need to check that the following is a well-defined double cell in $\calK$,
$$
\xymatrix@R=1.8em{
HG'X\ar[d]|-{\scriptscriptstyle\bullet}\ar[r]^{\beta_{G'X}}\ar@{}[dr]|\le & K'G'X\ar[d]|-{\scriptscriptstyle\bullet}
\\
HGX\ar[r]_{\beta_{GX}}& K'GX
}
$$
This is a well-defined double cell by the vertical functoriality of $\beta$ applied to the arrow $\xymatrix@1{G'X\ar[r]|-{\scriptscriptstyle\bullet}& GX}$.

For the second restriction, we need to check that the following is a well-defined double cell in $\calK$,
$$
\xymatrix@R=1.8em{
K'FX\ar[d]|-{\scriptscriptstyle\bullet}\ar[r]^{K'\alpha_X}\ar@{}[dr]|\le & K'G'X\ar[d]|-{\scriptscriptstyle\bullet}
\\
KFX\ar[r]_{K\alpha_X}&KG'X.
}
$$
This follows from the horizontal functoriality of the vertical transformation $K'\le K$, applied to the arrow $\xymatrix@1{FX\ar[r]^{\alpha_X}&G'X}$.
\end{proof}

\begin{prop}
Given ordered groupoids $\calG$, $\calH$, and $\calK$ with double functors
$$
\xymatrix{\calG\ar@<.5ex>[r]^F\ar@<-.5ex>[r]_{G}&\calH\ar@<.5ex>[r]^H\ar@<-.5ex>[r]_K&\calK}
$$
and $\Lambda$-transformations $(\alpha,\le)\colon F\Rightarrow G$ and $(\beta,\le)\colon H\Rightarrow K$. Then,
$$
(K\alpha,\le)\cdot (\beta F,\le)=(\beta G,\le)\cdot (H\alpha,\le)
$$
where $\cdot$ denotes vertical composition.
\end{prop}

\begin{proof}
Let $X$ be an object of $\calG$. Then the component of $K(\alpha,\le)\cdot (\beta,\le)F$ at $X$ is obtained by considering the following diagram in $\calK$:
$$
\xymatrix@C=4em@R=1.8em{
HFX\ar[r]^{\beta_{FX}}&K'FX\ar[d]|-{\scriptscriptstyle\bullet} \ar[r]^{K\alpha_X|_{K'FX}} & AX\ar[d]|-{\scriptscriptstyle\bullet}
\\
&KFX\ar[r]_{K\alpha_X}& KG'X\ar[d]|-{\scriptscriptstyle\bullet}
\\
&&KGX}
$$
By Lemma \ref{middle-four-tricks}, $K\alpha_X|_{K'FX}=K'\alpha_X$ and hence, $AX=K'G'X$. We conclude that the component of $K(\alpha,\le)\cdot (\beta,\le)F$ at $X$ is $(K'\alpha_X\circ\beta_{FX},\le)$.

The component of $(\beta,\le)G\cdot H(\alpha,\le)$ at $X$ is calculated as follows:
$$
\xymatrix@C=4em@R=1.8em{
HFX\ar[r]^{H\alpha_X}&HG'X\ar[d]|-{\scriptscriptstyle\bullet}\ar@{}[dr]|\le\ar[r]^-{\beta_{GX}|_{HG'X}} & BX\ar[d]|-{\scriptscriptstyle\bullet}
\\
&HGX\ar[r]_{\beta_{GX}}&K'GX\ar[d]|-{\scriptscriptstyle\bullet}
\\
&&KGX
}$$
By Lemma \ref{middle-four-tricks}, $\beta_{GX}|_{HG'X}=\beta_{G'X}$ and hence, $BX=K'G'X$. So the component of $(\beta,\le)G\cdot H(\alpha,\le)$ at $X$ is
$(\beta_{G'X}\circ H\alpha_X,\le)$.
Finally, note that $K'\alpha_X\circ\beta_{FX}=\beta_{G'X}\circ H\alpha_X$
by ordinary middle-four for horizontal transformations.
The result of the lemma now follows.
\end{proof}

\begin{prop}
Horizontal and vertical composition of $\Lambda$-transformations as defined above satisfy the middle-four interchange law.
\end{prop}

\begin{proof}
Consider the following double functors and $\Lambda$-cells between ordered groupoids:
$$
\xymatrix@C=4em{
\calG \ar@<4ex>[r]^F\ar@<2ex>@{}[r]|{\Downarrow(\alpha,\le)}\ar[r]|{G}\ar@<-2ex>@{}[r]|{\Downarrow(\gamma,\le)}\ar@<-4ex>[r]_H
 & \calH  \ar@<4ex>[r]^K\ar@<2ex>@{}[r]|{\Downarrow(\beta,\le)}\ar[r]|{L}\ar@<-2ex>@{}[r]|{\Downarrow(\delta,\le)}\ar@<-4ex>[r]_M&\calK
}
$$
We first calculate a part of $((\delta,\le)\cdot(\beta,\le))\circ((\gamma,\le)\cdot(\alpha,\le))$ and $((\delta\circ\gamma)\cdot(\beta\circ\alpha))$ respectively, using the results from Lemma \ref{middle-four-tricks}:
$$
\xymatrix@R=1.8em{
KFX\ar[r]^{K\alpha_X}&KG'X \ar[d]|-{\scriptscriptstyle\bullet}
\\
& KGX\ar[r]^{K\gamma_X}&KH'X \ar[d]|-{\scriptscriptstyle\bullet}\ar[r]^{\beta_{H'X}}\ar@{}[dr]|\le & L'H'X \ar[d]|-{\scriptscriptstyle\bullet}
\\
&&KHX\ar[r]_{\beta_{HX}}& L'HX \ar[d]|-{\scriptscriptstyle\bullet}
\\
&&& LHX\ar[r]_{\delta_{HX}}&M'HX \ar[d]|-{\scriptscriptstyle\bullet}
\\
&&&& MHX
}
$$
and
$$
\xymatrix@R=1.8em{
KFX\ar[r]^{K\alpha_X}&KG'X \ar[d]|-{\scriptscriptstyle\bullet}
\\
& KGX\ar[r]^{\beta_{GX}} & L'GX  \ar[d]|-{\scriptscriptstyle\bullet}\ar@{}[dr]|\le \ar[r]^{L'\gamma_X} & L'H'X  \ar[d]|-{\scriptscriptstyle\bullet}
\\
&&LGX\ar[r]_{L\gamma_X} & LH'X \ar[d]|-{\scriptscriptstyle\bullet}
\\
&&&LHX\ar[r]_{\delta_{HX}} & M'HX  \ar[d]|-{\scriptscriptstyle\bullet}
\\
&&&& MHX
}
$$
Note that $\beta_{H'X}\circ K\gamma_X=L'\gamma_X\circ\beta_{GX}$ by interchange for horizontal transformations. Hence, taking the remaining liftings in both diagrams will result in the same composites.
\end{proof}

\section{The Equivalence of 2-Categories}
\label{sec:the-equivalence-of-2-categories}

In this section we show that there is a 2-adjunction between the  2-category of ordered groupoids, double functors (ordered functors), and $\Lambda$-transformations and the
2-category of left-cancellative categories, functors and natural transformations.

As was observed by Lawson and Steinberg, any ordered groupoid of the form ${\mathbf G}(\calC)$ where $\calC$ is a left-cancellative category has maximal objects in the sense that each object is less than or equal to a unique maximal object. We will show that when we restrict ourselves to ordered groupoids with this property we obtain a biequivalence of 2-categories.

In our earlier introduction of the functors $\mathbf L$ and $\mathbf G$ we only gave their description on objects. We will now include their description on arrows (double functors and functors respectively) and then extend them to 2-functors; i.e., give their description on $\Lambda$-transformations and natural transformations respectively.

For a double functor $F\colon \calG\to\calH$, the functor $\mathbf{L}(F)\colon\mathbf{L}(\calG)\to\mathbf{L}(\calH)$ is on objects the same as $F$ and on arrows the extension is obvious: $\mathbf{L}(F)(h,\le)=(F(h),\le)$, and this is well-defined, since $F$ sends vertical arrows to vertical arrows, so it preserves the order relation. Now let $(\alpha,\le)\colon F\Rightarrow G$ be a $\Lambda$-transformation. Then $\mathbf{L}(\alpha,\le)\colon \mathbf{L}(F)\Rightarrow \mathbf{L}(G)$ is the natural transformation with components $(\alpha_X,\le)$.
In order to show that this is indeed natural, let $A\stackrel{h}{\rightarrow}B'\le B$ be an arrow in $\mathbf{L}(\calG)$.
Then we need to check that the following square commutes in $\mathbf{L}(\calH)$,
\begin{equation}\label{D:naturality}
\xymatrix@C=4em{
FA\ar[r]^{F(h,\le)}\ar[d]_{\mathbf{L}(\alpha,\le)_A} & FB\ar[d]^{\mathbf{L}(\alpha,\le)_B}
\\
GA\ar[r]_{G(h,\le)}&GB
}
\end{equation}
The composition $\mathbf{L}(\alpha,\le)_B\circ F(h,\le)$ is calculated as follows (in $\calH$):
$$
\xymatrix{
FA\ar[r]^{Fh}&FB'\ar[d]|-{\scriptscriptstyle\bullet}\ar[r]^{\alpha_{B'}}\ar@{}[dr]|\le & G'B'\ar[d]|-{\scriptscriptstyle\bullet}
\\
&FB\ar[r]_{\alpha_B} & G'B \ar[d]|-{\scriptscriptstyle\bullet}
\\
&&GB
}
$$
So $\mathbf{L}(\alpha,\le)_B\circ F(h,\le)=(\alpha_{B'}\circ Fh,\le)$.

The composition $G(h,\le)\circ\mathbf{L}(\alpha,\le)$ is calculated as follows:
$$
\xymatrix{
FA\ar[r]^{\alpha_A}&G'A\ar[d]|-{\scriptscriptstyle\bullet}\ar[r]^{G'h}\ar@{}[dr]|\le & G'B'\ar[d]|-{\scriptscriptstyle\bullet}
\\
&GA\ar[r]_{Gh} & GB' \ar[d]|-{\scriptscriptstyle\bullet}
\\
&&GB
}
$$
So $G(h,\le)\circ\mathbf{L}(\alpha,\le)=(G'h\circ\alpha_A,\le)$.
Now $\alpha_{B'}\circ Fh=G'h\circ\alpha_A$ by horizontal naturality of $\alpha$, so the naturality square (\ref{D:naturality}) for $\mathbf{L}(\alpha,\le)$ commutes.

It is straightforward to check that $\mathbf{L}$ preserves horizontal and vertical composition of 2-cells.

In the other direction, the functor $\mathbf{G}$ sends a functor $K\colon \calC\to\calD$ between left-cancellative categories to the double functor
$\mathbf{G}(K)\colon\mathbf{G}(\calC)\to\mathbf{G}(\calD)$ which is defined as follows. On objects, $\mathbf{G}(K)\colon [m\colon A\to B]
\mapsto [Km\colon KA\to KB]$. On horizontal arrows, $\mathbf{G}(K)([m,n])=[Km,Kn]$. It is straightforward to check that this is well-defined on
equivalence classes and preserves composition and identities.  Furthermore, since $K$ sends subobjects to subobjects, $\mathbf{G}(K)$ sends vertical
arrows to well-defined vertical arrows and double cells to double cells.

Now let $\theta\colon K\Rightarrow K'$ be a natural transformation.
Then the $\Lambda$-transformation $\mathbf{G}(\theta)$ has components
given by $$\mathbf{G}(\theta)_{[m]}=\left(\xymatrix@1@C=8em{[Km]\ar[r]^-{[Km,K'm\circ\theta_{\mbox{\scriptsize dom}(m)}]}&[K'm\circ\theta_{\mbox{\scriptsize dom}(m)}]\le[K'm]}\right).$$

To check that this is well-defined we need to show three things. First that the assignment $[m]\mapsto [K'm\circ\theta_{\mbox{\scriptsize dom}(m)}]$
on objects extends to a functor $\mathbf{G}(\calC)\to\mathbf{G}(\calD)$; call this functor $T$.
Second that the arrows $[Km,K'm\circ\theta_{\mbox{\scriptsize dom}(m)}]$ form the components of a horizontal transformation $\mathbf{G}(K)\Rightarrow_h T$
and third that the $[K'm\circ\theta_{\mbox{\scriptsize dom}(m)}]\le[K'm]$ form the components of a vertical transformation $T\Rightarrow_v \mathbf{G}(K')$.

To extend the definition of $T$ to horizontal arrows, note that for $[m,n]\colon[m]\to[n]$ in $\mathbf{G}(\calC)$, we have that
$\mbox{dom}(m)=\mbox{dom}(n)$, so $[m,n]\mapsto[K'm\circ\theta_{\mbox{\scriptsize dom}(m)},K'n\circ\theta_{\mbox{\scriptsize dom}(n)}]$ is
well-defined as far as shape is concerned.
To check that it is well-defined on equivalence classes and that this assignment preserves the partial order, consider the following commutative diagram in $\calC$:
\begin{equation}\label{DP:le1}
\xymatrix@R=.8em@C=3em{
&A\ar[dd]_\ell\ar[dl]_m\ar[dr]^n
\\
B&&C
\\
&A'\ar[ul]^{m'}\ar[ur]_{n'}
}
\end{equation}
This gives rise to the following commutative diagram in $\calD$:
\begin{equation}\label{DP:le2}
\xymatrix@R=.8em@C=3em{
&&KA\ar[dl]_{\theta_A}\ar[dr]^{\theta_A}\ar[4,0]|{K\ell}
\\
&K'A\ar[dd]|{K'\ell}\ar[dl]_{K'm}&&K'A\ar[dd]|{K'\ell}\ar[dr]^{K'n}
\\
K'B&&&&K'C
\\
&K'A'\ar[ul]^{K'm'}&&K'A'\ar[ur]_{K'n'}
\\
&&KA'\ar[ul]^{\theta_{A'}}\ar[ur]_{\theta_{A'}}
}
\end{equation}
This shows that when $[m,n]=[m',n']$ (i.e., when $\ell$ is an isomorphism), then
$T[m,n]=T[m',n']$ (since $K\ell$ is then an isomorphism as well).
Furthermore, it shows that for any double cell
$$
\xymatrix@C=3em@R=1.5em{
[m]\ar[r]^{[m,n]}\ar[d]|-{\scriptscriptstyle\bullet}\ar@{}[dr]|\le&[n]\ar[d]|-{\scriptscriptstyle\bullet}
\\
[m']\ar[r]_{[m',n']}&[n']
}
$$
in $\mathbf{G}(\calC)$
(corresponding to the existence of an arbitrary arrow $\ell$ in (\ref{DP:le1}))
there is a corresponding double cell in $\mathbf{G}(\calD)$,
$$
\xymatrix@C=3em@R=1.5em{
T[m]\ar[d]|-{\scriptscriptstyle\bullet}\ar[r]^{T[m,n]}\ar@{}[dr]|\le & T[n]\ar[d]|-{\scriptscriptstyle\bullet}
\\
T[m']\ar[r]_{T[m',n']}&T[n']
}
$$
(corresponding to $K\ell$ in (\ref{DP:le2})). So $T$ can be extended to a double functor $\mathbf{G}(\calC)\to\mathbf{G}(\calD)$.

The proof that the $t_{[m]}=[Km,K'm\circ\theta_{\mbox{\scriptsize dom}(m)}]$ form the components of a horizontal transformation
$t\colon \mathbf{G}(K)\Rightarrow_h T$ (i.e., that they satisfy horizontal naturality and vertical functoriality) is completely straightforward.
The same is true for the proof that the $T[m]\le K'[m]$ form the components of a vertical transformation.

We finally need to check that $\mathbf{G}$ thus defined preserves horizontal and vertical composition of 2-cells.
For vertical composition, suppose we have natural transformations $\xymatrix@1{K\ar@{=>}[r]^\theta &  K'\ar@{=>}[r]^{\theta'}&K''}$.
Then the component of $\mathbf{G}(\theta')\cdot\mathbf{G}(\theta)$ at $[m]$ is calculated as follows,
$$
\xymatrix@C=6.75em{
[Km]\ar[r]^-{t_{[m]}} &\ar[d]|-{\scriptscriptstyle\bullet} [K'm\circ\theta_{\mbox{\scriptsize dom}(m)}]
\ar[rr]^-{[K'm\theta_{\mbox{\scriptsize dom}(m)},K''m\theta'_{\mbox{\scriptsize dom}(m)}\theta_{\mbox{\scriptsize dom}(m)}]}\ar@{}[drr]|-\le
&&[K''m\circ\theta'_{\mbox{\scriptsize dom}(m)}\theta_{\mbox{\scriptsize dom}(m)}]\ar[d]|-{\scriptscriptstyle\bullet}
\\
&[K'm]\ar[rr]_-{t'_{[m]}}&&\ar[d]|-{\scriptscriptstyle\bullet} [K''m\circ\theta'_{\mbox{\scriptsize dom}(m)}]
\\
&&&[K''m]
}
$$
The composition of the top two horizontal arrows is
$[Km,K''m\theta'_{\mbox{\scriptsize dom}(m)}\theta_{\mbox{\scriptsize dom}(m)}]=\mathbf{G}(\theta'\cdot\theta)$, as required.

Since vertical composition is preserved, it is sufficient to check that whiskering is preserved in order to obtain preservation of horizontal composition.
This is a straightforward calculation and left to the reader.

\begin{thm}\label{2-adjunction}
The 2-functors $\mathbf{L}\colon \mbox{\textbf{\textit{oGpd}}}\to\mbox{\textbf{\textit{lcCat}}}$ and
$\mathbf{G}\colon\mbox{\textbf{\textit{lcCat}}}\to\mbox{\textbf{\textit{oGpd}}}$ define a 2-adjunction,
$$\mbox{\textbf{\textit{oGpd}}}\simeq \mbox{\textbf{\textit{lcCat}}}.$$
\end{thm}

\begin{proof}
In order to prove this we will show that the functors $\eta_{\calC}$ from Section \ref{D:eta} form a strong natural transformation of 2-functors
$\eta\colon\mbox{Id}_{\mbox{\textbf{\textit{\scriptsize lcCat}}}}\Rightarrow\mathbf{L}\mathbf{G}$ and the double functors $\kappa_\calG$
from Section \ref{D:kappa} form a strong natural
transformation $\kappa\colon\mathbf{G}\mathbf{L}\Rightarrow \mbox{Id}_{\mbox{\textbf{\textit{\scriptsize oGpd}}}}$
(i.e., all naturality squares commute on the nose).
Furthermore we will show that the triangle identity diagrams for $\eta$ and $\kappa$ commute on the nose as well.

To consider the naturality for $\eta$, let $F\colon\calC\to\calD$ be a functor between left-cancellative categories.
Then the naturality square for $F$ is
$$
\xymatrix{
\calC\ar[d]_F\ar[r]^-{\eta_{\calC}}&\mathbf{L}\mathbf{G}(\calC)\ar[d]^{\mathbf{L}\mathbf{G}(F)}
\\
\calD\ar[r]_-{\eta_{\calD}}&\mathbf{L}\mathbf{G}(\calD).
}$$
The composition $\mathbf{L}\mathbf{G}(F)\circ\eta_{\calC}$ gives on objects,
$$
A\mapsto[1_A]\mapsto[F1_A]=[1_{FA}]
$$
and on arrows,
$$
\left(A\stackrel{f}{\longrightarrow}B\right)\mapsto \left([1_A]\stackrel{[f]}{\longrightarrow}[1_B]\right)\mapsto \left([1_{FA}]\stackrel{[Ff]}{\longrightarrow}[1_{FB}]\right).
$$
The other composition, $\eta_{\calD}\circ F$, gives on objects,
$$A\mapsto FA\mapsto [1_{FA}]$$
and on arrows,
$$
 \left(A\stackrel{f}{\longrightarrow}B\right)\mapsto \left(FA\stackrel{Ff}{\longrightarrow}FB\right)\mapsto \left([1_{FA}]\stackrel{[Ff]}{\longrightarrow}[1_{FB}]\right).
$$
We conclude that the naturality square commutes on the nose.

To consider the naturality for $\kappa$, let $\varphi\colon \calG\to\calH$ be a double functor.
Then the naturality square becomes
$$
\xymatrix{
\mathbf{G}\mathbf{L}(\calG)\ar[d]_{\mathbf{G}\mathbf{L}(\varphi)}\ar[r]^-{\kappa_\calG} & \calG\ar[d]^\varphi
\\
\mathbf{G}\mathbf{L}(\calH)\ar[r]^-{\kappa_\calH} & \calH}
$$
To show that this square commutes, we will check what each of the composites does with a double cell in $\mathbf{G}\mathbf{L}(\calG)$
and its domains and codomains.

A general double cell in $\mathbf{G}\mathbf{L}(\calG)$ is of the form
\begin{equation}\label{gnrldblcell}
\xymatrix@R=1.5em{
(B', B)\ar[d]|-{\scriptscriptstyle\bullet}\ar@{}[dr]|\le\ar[r]^h&(C', C)\ar[d]|-{\scriptscriptstyle\bullet}
\\
(D', B)\ar[r]_k& (E', C)
}
\end{equation}
where
\begin{equation}
\label{induceddblcell}
\xymatrix@R=1.5em{
 B'\ar[d]|-{\scriptscriptstyle\bullet}\ar@{}[dr]|\le\ar[r]^h&C'\ar[d]|-{\scriptscriptstyle\bullet}
\\
D'\ar[r]_k& E'}
\end{equation}
is a double cell in $\calG$.
The double functor $\kappa_\calG$ sends (\ref{gnrldblcell}) to (\ref{induceddblcell}) and $\varphi$ sends  (\ref{induceddblcell})  to the double cell
\begin{equation}
\label{imagedblcell}
\xymatrix@R=1.5em{
 \varphi B'\ar[d]|-{\scriptscriptstyle\bullet}\ar@{}[dr]|\le\ar[r]^{\varphi h}&\varphi C'\ar[d]|-{\scriptscriptstyle\bullet}
\\
\varphi D'\ar[r]_{\varphi k}& \varphi E'}
\end{equation}
in $\calH$.
For the other composition, $\mathbf{G}\mathbf{L}(\varphi)$ sends (\ref{gnrldblcell}) to
\begin{equation}\label{secondimage}
\xymatrix@R=1.5em{
(\varphi B', \varphi B)\ar[d]|-{\scriptscriptstyle\bullet}\ar@{}[dr]|\le\ar[r]^{\varphi h}&(\varphi C', \varphi C)\ar[d]|-{\scriptscriptstyle\bullet}
\\
(\varphi D', \varphi B)\ar[r]_{\varphi k}& (\varphi E', \varphi C)
}
\end{equation}
and $\kappa_\calH$ sends (\ref{secondimage}) to (\ref{imagedblcell}) as required for commutativity.

We will now check the triangle identities,
$$
\xymatrix@C=4em{
\mathbf{G}\ar@{=>}[r]^{\mathbf{G}\eta}\ar@{=}[dr] & \mathbf{G}\mathbf{L}\mathbf{G}\ar@{=>}[d]^{\kappa\mathbf{G}}
  & \mathbf{L}\ar@{=>}[r]^{\eta\mathbf{L}}\ar@{=}[dr] & \mathbf{L}\mathbf{G}\mathbf{L}\ar@{=>}[d]^{\mathbf{L}\kappa}
\\
&\mathbf{G}&&\mathbf{L}
}
$$
For the first triangle, we check the components at a left-cancellative category $\calC$,
$$
\xymatrix{
\mathbf{G}(\calC)\ar[r]^{\mathbf{G}\eta_\calC}\ar@{=}[dr]&\mathbf{G}\mathbf{L}\mathbf{G}(\calC)\ar[d]^{\kappa_{\mathbf{G}\calC}}
\\
&\mathbf{G}(\calC)
}
$$
So we calculate the composition of double functors, $(\kappa_{\mathbf{G}\calC})\circ(\mathbf{G}\eta_\calC)$.
A typical double cell in $\mathbf{G}(\calC)$ is of the form,
$$
\xymatrix@C=4em@R=1.5em{
[n]\ar[d]|-{\scriptscriptstyle\bullet}\ar@{}[dr]|\le\ar[r]^{[n,m]}& [m]\ar[d]|-{\scriptscriptstyle\bullet}
\\
[n']\ar[r]_{[n',m']}& [m']
}
$$
corresponding to a commutative diagram in $\calC$ of the form,
$$
\xymatrix@R=.8em@C=3em{
&A\ar[dd]|v\ar[dl]_{n}\ar[dr]^{m}
\\
B&&C
\\
&A'\ar[ul]^{n'}\ar[ur]_{m'}
}
$$
It's image under $\mathbf{G}\eta_\calC$
is
$$
\xymatrix@C=4.5em@R=1.5em{
[[1_A]\stackrel{[n]}{\longrightarrow}[1_B]]\ar[r]^{[[n],[m]]}\ar[d]|-{\scriptscriptstyle\bullet}\ar@{}[dr]|\le & [[1_A]\stackrel{[m]}{\longrightarrow}[1_C]]\ar[d]|-{\scriptscriptstyle\bullet}
\\
[[1_{A'}]\stackrel{[n']}{\longrightarrow}[1_B]]\ar[r]_{[[n'],[m']]} & [[1_{A'}]\stackrel{[m']}{\longrightarrow}[1_C]]
}
$$
This is the same as
$$
\xymatrix@C=3.5em@R=1.5em{([n], [1_B])\ar[r]^{[1_A]}\ar[d]|-{\scriptscriptstyle\bullet} \ar@{}[dr]|\le& ([m], [1_C])\ar[d]|-{\scriptscriptstyle\bullet}
\\
([n'], [1_B])\ar[r]_{[1_{A'}]} & ([m'], [1_C])}
$$
Now $\kappa_{\mathbf{G}\calC}$ maps this to
$$
\xymatrix@C=3.5em@R=1.5em{
[n]\ar[d]|-{\scriptscriptstyle\bullet}\ar@{}[dr]|\le\ar[r]^{[n,m]}& [m]\ar[d]|-{\scriptscriptstyle\bullet}
\\
[n']\ar[r]_{[n',m']}& [m']
}
$$
and we see that $(\kappa_{\mathbf{G}\calC})\circ(\mathbf{G}\eta_\calC)=1_{\calC}$ as required.

To verify the other triangle identity, let $\calG$ be an ordered groupoid.
Then an arrow of $\mathbf{L}(\calG)$ is of the form $A\stackrel{h}{\longrightarrow}B\le C$ (a formal composite of a horizontal and vertical arrow in $\calG$).
Now $\eta_{\mathbf{L}(\calG)}$ send this arrow to
$$
\xymatrix@C=3em{
[1_{A}]\ar[r]^{[h,\le]}&[1_{C}]
}
$$
and this is equivalent to
$$
[\xymatrix{A\ar[r]|\bullet & A}]\stackrel{[h]}{\longrightarrow}[\xymatrix{B\ar[r]|\bullet & C}]\le [\xymatrix{C\ar[r]|\bullet & C}]
$$
where the latter is seen as the formal composite of a horizontal and vertical arrow in $\mathbf{G}\mathbf{L}(\calG)$.
Note that $\mathbf{L}\kappa_\calG$ sends this composite to the formal composite $A\stackrel{h}{\longrightarrow}B\le C$, as required.

This concludes the proof of Theorem \ref{2-adjunction}.
\end{proof}

We write $\mbox{\textbf{\textit{oGpd}}}_{\mbox{\scriptsize max}}$ for the full sub-2-category of ordered groupoids with maximal objects. (Note that the morphisms in this category need not send maximal objects to maximal objects.)
As was noticed by Lawson and Steinberg in Section 2.1 of \cite{L1},
the functor $\mathbf G$ sends each left-cancellative category to an object of $\mbox{\textbf{\textit{oGpd}}}_{\mbox{\scriptsize max}}$. It is also easy to see that the restricted functor ${\mathbf L}\colon \mbox{\textbf{\textit{oGpd}}}_{\mbox{\scriptsize max}}\to\mbox{\textbf{\textit{lcCat}}}$ is still essentially surjective on objects. With this restriction we obtain an equivalence of 2-categories.

\begin{cor}\label{2-equivalence}
The 2-functors $\mathbf{L}\colon \mbox{\textbf{\textit{oGpd}}}_{\mbox{\scriptsize max}}\to\mbox{\textbf{\textit{lcCat}}}$ and
$\mathbf{G}\colon\mbox{\textbf{\textit{lcCat}}}\to\mbox{\textbf{\textit{oGpd}}}_{\mbox{\scriptsize max}}$ define a 2-adjoint biequivalence,
$$\mbox{\textbf{\textit{oGpd}}}_{\mbox{\scriptsize max}}\simeq \mbox{\textbf{\textit{lcCat}}}.$$
\end{cor}

\begin{proof}
The components of both $\eta$ and $\kappa$ are essential equivalences (of categories and ordered groupoids respectively).
In order to get a biequivalence we need to show that these components have pseudo inverses. To obtain a pseudo inverse for $\eta_{\calC}$, we need to choose
a representative $(\bar{m}\colon \bar{A}_m\to {B})$ for each subobject $[m\colon A\to B]$.
Then each arrow $[h]\colon [m\colon A\to B]\to[m'\colon A'\to B']$ has precisely one
representative $\bar{h}_{m,m'}\colon\bar{A}_m\to \bar{A}'_{m'}$ such that
$[\bar{h}_{m,m'}]\colon [\bar{m}\colon\bar{A}_m\to{B}]\to[\bar{m'}\colon \bar{A}'_{m'}\to {B}']$ is the same as $[h]\colon[m\colon A\to B]\to[n\colon A'\to B']$. So a pseudo inverse of $\eta_\calC$ can be defined by sending an object $[m\colon A\to B]$ to $\bar{A}_m$ and an arrow $[h]\colon[m]\to[m']$ to $\bar{h}_{m,m'}\colon\bar{A}_m\to\bar{A}'_{m'}]$.

To define a pseudo inverse for $\kappa_{\calG}$, write $\hat{A}$ for the maximal object with $A\le \hat{A}$ in $\calG$. Note that when $A\le A'$, then $\hat{A}=\hat{A}'$.
Then a pseudo inverse for $\kappa_\calG$ is given by the assignment
$$
\xymatrix{
B\ar[d]|-{\scriptscriptstyle\bullet}\ar[r]^h\ar@{}[dr]|\le&C\ar[d]|-{\scriptscriptstyle\bullet}\ar@{}[drr]|{\textstyle\mapsto}&&(B,\hat{B})\ar[d]|-{\scriptscriptstyle\bullet}\ar[r]^h\ar@{}[dr]|\le&(C,\hat{C})\ar[d]|-{\scriptscriptstyle\bullet}
\\
D\ar[r]_{k}&E&&(D,\hat{B})\ar[r]_{k}&(E,\hat{C})
}
$$

\end{proof}

\section{Applications}
\label{sec:applications}
\subsection{Presheaves on Ordered Groupoids}\label{D:presheaves}
In terms of double categories, presheaves on ordered groupoids as defined in \cite{lawson2004} can be described similarly to presheaves on ordinary categories. The role of the category {\bf Set} is now taken by the double category ${\mathbb Q}{\bf Set}$ of quartets in the category of sets (as defined by Ehresmann): the objects of  ${\mathbb Q}{\bf Set}$ are sets, the horizontal and vertical arrows are functions and the double cells are commutative squares in {\bf Set}.
Then a presheaf $F$ on an ordered groupoid $\calG$ is a functor
\begin{equation}\label{ogro-presheaf}
F\colon \calG^{\mbox{\scriptsize op,op}}\to{\mathbb Q}\mbox{\bf Set},
\end{equation}
which is contravariant in both the horizontal and vertical direction and sends double cells to commutative squares. Note that by the symmetry of the double category ${\mathbb Q}\mbox{\bf Set}$, horizontal and vertical transformations between such presheaf functors amount to the same thing:  a collection of functions $\alpha_A\colon FA\to F'A$  which is natural in $A$ both when considered with respect to horizontal arrows and with respect to vertical arrows. The category $\mbox{\bf PreSh}(\calG)$ is then defined as the category of double functors as in (\ref{ogro-presheaf})  with these transformations as arrows.

Lawson and Steinberg \cite{lawson2004} show that there is an isomorphism of categories
\begin{equation}\label{LS-presheaves}
\mbox{\bf PreSh}(\calG)\cong\mbox{\bf PreSh}({\mathbf L}(\calG)).
\end{equation}
Furthermore, since any weak equivalence of categories induces an equivalence between the corresponding presheaf categories,
we have
$$
\mbox{\bf PreSh}(\calC)\simeq\mbox{\bf PreSh}(\mathbf{LG}(\calC)),
$$
and by combining these equivalences we obtain,

\begin{equation}
\label{eqn:presheaves-c-gc}
\mbox{\bf PreSh}({\mathbf G}(\calC))\simeq\mbox{\bf PreSh}(\calC)\mbox{ and }\mbox{\bf PreSh}(\calG)\simeq\mbox{\bf PreSh}(\mathbf{GL}(\calG)).
\end{equation}

We will now provide an explicit description of the functors that give the
equivalence
\(\mbox{\bf PreSh}({\mathbf G}(\calC))\simeq\mbox{\bf PreSh}(\calC)\),
which will be needed when proving
Proposition \ref{prop:sheaves-G-and-LG-and-C-and-GC}.
\begin{proposition}
\label{prop:presheaves-c-gc}
The equivalence of categories
\(\mbox{\bf PreSh}({\mathbf G}(\calC))\simeq\mbox{\bf PreSh}(\calC)\)
is given by a pair of functors
\[
  \xymatrix{
  \mbox{\bf PreSh}({\mathbf G}(\calC)) \ar@<+0.5pc>[r]^-{\check{(-)}} &
  \mbox{\bf PreSh}(\calC) \ar@<+0.5pc>[l]^-{\tilde{(-)}}
  }
\]
\end{proposition}
\begin{proof}
Let $\Phi$ be a presheaf on $\calC$.
To define the corresponding presheaf $\tilde\Phi$ on $\bG(\calC)$, we need to make some choices.
For each objects $[m\colon A\to B]$ in $\bG(\calC)$, choose a representative $[\bar{m}\colon\bar{A}_m\to B]$.
For each horizontal arrow $\xymatrix@1{[m\colon A\to B]\ar[r]^{[m,n]}&[n\colon A\to C]}$
there are unique arrows $\mu_m\colon A\to \bar{A}_{m}$ and $\mu_n\colon A\to \bar{A}_n$ and we write
$$
\xymatrix@C=4em{[\bar{m}] \ar[r]^{\langle\mu_n\circ\mu_m^{-1}\rangle}&[\bar{n}]}$$
for the arrow $[m,n]$.
For a vertical arrow $\xymatrix{[n_1]\ar[r]|{\scriptscriptstyle\bullet}&[n_2]}$,
there is a unique arrow $v_{n_1,n_2}$ in $\calC$ such that
$\bar{n}_2=\bar{n}_1v_{n_1,n_2}$, so we label the vertical arrow as
$$\xymatrix@C=4em{[\bar{n}_1]\ar[r]|{\scriptscriptstyle\bullet}^{\{v_{n_1,n_2}\}}&[\bar{n_2}}]$$
The reader may check that if
$$
\xymatrix@C=4em{
[n_1\colon X\to B]\ar[d]|{\scriptscriptstyle\bullet}\ar[r]^{[n_1,n_2]}\ar@{}[dr]|{\le}& [n_2\colon X\to C]\ar[d]|{\scriptscriptstyle\bullet}
\\
[m_1\colon Y\to B]\ar[r]_{[m_1,m_2]}&[m_2\colon Y\to C]}
$$
is a double cell in $\bG(\calC)$, then the corresponding square
$$
\xymatrix@C=4em{
[\bar{n}_1\colon \bar{X}_{n_1}\to B]\ar[d]|{\scriptscriptstyle\bullet}_{\{v_{n_1,m_1}\}}\ar@{}[dr]|{\le} \ar[r]^{\langle\mu_{n_2}\mu_{n_1}^{-1}\rangle} & [\bar{n}_2\colon \bar{X}_{n_2}\to C]\ar[d]|{\scriptscriptstyle\bullet}^{\{v_{n_2,m_2}\}}
\\
[\bar{m}_1\colon \bar{Y}_{m_1}\to B]\ar[r]_{\langle\mu_{m_2}\mu_{m_1}^{-1}\rangle}&[\bar{m}_1\colon \bar{Y}_{m_1}\to C]}
$$
gives rise to a commutative square in $\calC$:
$$
\xymatrix@C=5em{
\bar{X}_{n_1}
\ar[d]_{v_{n_1,m_1}}
\ar[r]^{\mu_{n_2}\mu_{n_1}^{-1}} & \bar{X}_{n_2}
\ar[d]^{v_{n_2,m_2}}
\\
\bar{Y}_{m_1}\ar[r]_{\mu_{m_2}\mu_{m_1}^{-1}}&\bar{Y}_{m_1}}
$$

The corresponding presheaf $\tilde\Phi$ on $\bG(\calC)$ is then defined by
\begin{itemize}
\item On objects:
$\tilde\Phi([m\colon A\to B])=\Phi(\bar{A}_{m})$;
\item
On a horizontal arrow $\xymatrix@1@C=3em{[m\colon A\to B]\ar[r]^{[m,n]}&[n\colon A\to C]}$, define $\tilde\Phi([m,n])=\Phi(\mu_m\mu_n^{-1})\colon \Phi(\bar{A}_n)\to\Phi(\bar{A}_m)$.
\item
On a vertical arrow $\xymatrix@1@C=3em{[m\colon A\to B]\ar[r]|{\scriptscriptstyle\bullet}^{(m,m')}&[m'\colon A'\to B]}$, define $\tilde\Phi((m,n))=\Phi(v_{m,m'})\colon \Phi(\bar{A}'_{m'})\to\Phi(\bar{A}_m)$.
\end{itemize}
We leave it to the reader to verify that this gives a well-defined presheaf on $\bG(\calC)$.

In the opposite direction, let $\Psi$ be a presheaf on $\bG(\calC)$.
Then define the presheaf $\check\Psi$ on $\calC$ by
$\check\Psi(X)=\Psi[1_X\colon X\to X]$.
For an arrow $f\colon X\to Y$ in $\calC$, we define $\check\Psi(f)$ as the composite of the images under $\Psi$ of the arrows in the diagram,
$$\xymatrix{
[1_A]\ar[r]^{[1,f]}&[f]\ar[d]|{\scriptscriptstyle\bullet}^{v_f}\\&[1_B] }
$$
\end{proof}

In particular, we see that a presheaf topos is an \'etendue if and only if it can be presented as presheaves on an ordered groupoid. The following argument
shows that the isomorphisms and equivalences from
Equations (\ref{LS-presheaves}) and (\ref{eqn:presheaves-c-gc})
are the components of natural transformations
\(\mathbf{PreSh}(-)\Rightarrow\mathbf{PreSh}(\bL(-))\) and
\(\mathbf{PreSh}(-)\Rightarrow\mathbf{PreSh}(\bG(-))\).

\begin{rmk}\label{D:arrowcorrespondence}
The following arguments show that when ordered groupoid morphisms and functors of left-cancellative categories correspond to each other under the biequivalence given in Corollary \ref{2-equivalence}, they produce suitably isomorphic morphisms of presheaf categories.
\begin{enumerate}
\item
By just spelling out the definitions, we see that for a morphism $M\colon\calG\to\calG'$, the induced functor between presheaf toposes is the same as
the one induced by its \(\bL\)-image, $\bL(M)\colon\bL(\calG)\to \bL(\calG')$, in the sense that the following diagram commutes,
  \[
  \xymatrix@C=5pc{
  \mathbf{PreSh}(\calG')
    \ar[r]^{M^*}
    \ar[d]_{\cong}
    &
  \mathbf{PreSh}(\calG)
    \ar[d]^{\cong}
  \\
  \mathbf{PreSh}(\bL(\calG'))
    \ar[r]_{(\bL(M))^*} &
  \mathbf{PreSh}(\bL(\calG))
  }
  \]
where the vertical isomorphisms are the ones from (\ref{LS-presheaves}).
\item
The biequivalence
  \(
  \mbox{\textbf{\textit{oGpd}}}_{\mbox{\scriptsize max}}\simeq
  \mbox{\textbf{\textit{lcCat}}}
  \)
 induces the following diagram
$$
\xymatrix@C=4em{
\calC\ar[r]^F\ar[d]_{\simeq}\ar@{}[dr]|\cong & \calC'\ar[d]^{\simeq}
\\
\bL\bG(\calC)\ar[r]_{\bL\bG(F)}&\bL\bG(\calC').
}
$$
Combining this with the result in the first point of this remark, we obtain the following,
\[
\xymatrix@C=5pc{
  \mathbf{PreSh}(\calC')
    \ar[r]^{F^*}
    \ar[d]_{\simeq}
    \ar@{}[rd]|{\cong}&
  \mathbf{PreSh}(\calC)
    \ar[d]^{\simeq}
  \\
  \mathbf{PreSh}(\bL\bG(\calC'))
    \ar[r]_{(\bL\bG(F))^*}
    \ar[d]_{\cong}
  &
  \mathbf{PreSh}(\bL\bG(\calC))
    \ar[d]^{\cong}
  \\
  \mathbf{PreSh}(\bG(\calC'))
    \ar[r]_{(\bG(F))^*} &
  \mathbf{PreSh}(\bG(\calC))
}
\]
\end{enumerate}
This means that $M^*$ has a particular property, such as being an equivalence of categories or being left exact, if and only if $(\bL(M))^*$ has it and similarly,
$F^*$ has a property if and only if $(\bG(F))^*$ has it.
\end{rmk}

\subsection{Sheaves on Ehresmann Sites}
In this section we review the concept of an Ehresmann topology and reformulate it in double categorical language. An Ehresmann topology as introduced by Lawson and Steinberg \cite{lawson2004} consists of an assignment of special order ideals (so-called {\em covering ideals}) of the poset $\downarrow A=\{A'\le A\}$ for each object $A$, satisfying a number of conditions. Since $\downarrow A$ is part of the vertical structure of the ordered groupoid as double category, we will call these order ideals {\em vertical sieves}.

In Lawson and Steinberg's presentation, the condition on a Grothendieck topology to be closed under pullback was matched by the condition that an Ehresmann topology be closed under a notion of `$\star$-conjugation'. In our set-up we will need the following notion.

\begin{notation}
For a vertical sieve $\mathcal B$ on an object $B$
and a diagram
$$
\xymatrix@R=1.5em@C=3em{
A\ar[r]^f & B'\ar[d]|-{\scriptscriptstyle\bullet}\\
&B}
$$
in an ordered groupoid $\calG$, we define $f^*{\mathcal B}$ to be the following vertical sieve on $A$,
$$
f^*{\mathcal B}=\left\{\xymatrix@1{A'\ar[r]|-{\scriptscriptstyle\bullet}&A}\,|\,\mbox{cod}(f|_{A'})=B''\mbox{ with } \xymatrix@1{(B''\ar[r]|-{\scriptscriptstyle\bullet}&B)}\in{\mathcal B}\right\}.
$$
Note that $\mbox{cod}(f|_{A'})=B''$ if and only if there is a double cell
$$
\xymatrix{
A'\ar[d]|-{\scriptscriptstyle\bullet}\ar@{}[dr]|\le\ar[r]&B''\ar[d]|-{\scriptscriptstyle\bullet}
\\
A\ar[r]_f&B'
}
$$
\end{notation}
\begin{dfn}
An {\em Ehresmann topology} on an ordered groupoid $\calG$ is given by an assignment of a collection $T(A)$ of vertical sieves to each object $A$, such that:
\begin{itemize}
\item{(ET.1)}
The trivial sieve $(\downarrow A)\in T(A)$.
\item{(ET.2)}
If $\calB\in T(B)$ and $f\colon A\to B'$ with $\xymatrix@1{B'\ar[r]|-{\scriptscriptstyle\bullet}& B}$, then $f^*\calB\in T(A)$.
\item{(ET.3)}
Let $\calA\in T(A)$ and let $\calB$ be any vertical sieve on $A$.
If for each $$\xymatrix@C=3em@R=1.5em{C\ar[r]^{f}&A'\ar[d]|-{\scriptscriptstyle\bullet}\\&A}$$ with $\left(\xymatrix@1{A'\ar[r]|-{\scriptscriptstyle\bullet}&A}\right)\,\in\calA$,
$f^*\calB\in T(C)$, then $\calB\in T(A)$.
\end{itemize}
\end{dfn}

Lawson and Steinberg show that Grothendieck topologies on a left-cancellative category $\mathcal C$ are in one-to-one correspondence with Ehresmann topologies on $\mathbf{G}(\calC)$ and conversely, that Ehresmann topologies on
an ordered groupoid $\calG$ are in one-to-one correspondence with Grothendieck topologies on $\mathbf{L}(\calG)$. We summarize the correspondence in our notation.

Given an Ehresmann topology $T$ on an ordered groupoid $\calG$, the corresponding Grothendieck topology $J_T$ on $\mathbf{L}(\calG)$ is given by
$$\{\xymatrix{B_i\ar[r]^{m_i}&A_i'\ar[r]|-{\scriptscriptstyle\bullet}&A}|i\in I\}\in J_T(A)\mbox{ if and only if }
\{\xymatrix{A_i'\ar[r]|-{\scriptscriptstyle\bullet}&A}|i\in I\}\in T(A).
$$
Note that, given a Grothendieck topology $J$ on $\mathbf{L}(\calG)$, the corresponding Ehresmann topology $T_J$ on $\calG$ can be recovered as follows:
$T_J(A)$ consists of those vertical sieves $\calA$ such that
$$\{\xymatrix@1{A'\ar[r]^{1_{A'}}&A'\ar[r]|-{\scriptscriptstyle\bullet}&A}\,|\,\xymatrix@1{A'\ar[r]|-{\scriptscriptstyle\bullet}&A} \mbox{ in }\calA \}\in J(A).$$

Given a Grothendieck topology $J$ on a left-cancellative category $\calC$, the corresponding Ehresmann topology $T_J$ on $\mathbf{G}(\calC)$ is given by
$$T_J([m\colon A\to B])=\{[mS];S\in J(A)\},\mbox{ where }[mS]=\{\xymatrix{[mn]\ar[r]|-{\scriptscriptstyle\bullet}&[m]}|\,n\in S\}.$$
Furthermore, given an Ehresmann topology $T$ on $\mathbf{G}(\calC)$, the corresponding Grothendieck topology $J_T$ is defined by:
$$
\{m_i\colon A_i\to A|\,i\in I\}\in J_T(A)\mbox{ if and only if }\{\xymatrix@1{[m_i]\ar[r]|-{\scriptscriptstyle\bullet}&[1_A]}|\,i\in I\}\in T([1_A]).
$$

\begin{proposition}
\label{prop:sheaves-G-and-LG-and-C-and-GC}
\begin{enumerate}
  \item
  \label{toshow:shG-equiv-shLC}
  (Lawson and Steinberg)
  The category of sheaves on an Ehresmann site \((\calG, T)\)
  is equivalent to the category of sheaves on \((\bL(\calG), J_T)\).
  \item
  \label{toshow:shC-equiv-shGC}
  The category of sheaves on a left-cancellative site \((\calC, J)\)
  is equivalent to the category of sheaves on \((\bG(\calC), T_J)\).
\end{enumerate}
\end{proposition}

\begin{proof}
We want to show that the maps involved in the equivalences listed at the beginning of Section \ref{D:presheaves} send sheaves to sheaves.
For part \ref{toshow:shG-equiv-shLC}, this was established by Lawson and Steinberg in \cite[Theorem 4.4]{lawson2004}.

To prove part \ref{toshow:shC-equiv-shGC},
recall the functors given in the proof of Proposition \ref{prop:presheaves-c-gc}
giving an equivalence of presheaf categories
$$
\mathbf{PreSh}(\calC)\simeq \mathbf{PreSh}(\bG(\calC)).
$$
We need to show that if $\Phi$ is a sheaf, so is $\tilde\Phi$ and if $\Psi$
is a sheaf, so is $\check\Psi$.

So assume that $\Phi$ is a sheaf on the site $(\calC,J)$.
We want to show that $\tilde\Phi$ is a sheaf on the Ehresmann site
$(\bG(\calC),T_J)$. So let $\{\xymatrix@1{[h_i\colon A_i\to B]\ar[r]|{\scriptscriptstyle\bullet}&[h\colon A\to B]}|\,i\in I\}\in T_J([h])$ and let
$\varphi_i\in \tilde\Phi([h_i])$ with $i\in I$ be a matching family.
Then for each index $i$, $\varphi_i\in\Phi(\bar{A}_{i,h_i})$ and there is an arrow $k_i\colon \bar{A}_{i,h_i}\to\bar{A}_h$ that makes the following triangle commute,
$$
\xymatrix{
\bar{A}_{i,h_i}\ar[dr]_{\bar{h}_i}\ar[rr]^{k_i}&&\bar{A}_h\ar[dl]^{\bar{h}}
\\
&B
}
$$
Then it follows that $\{k_i|\,i\in I\}\in J(\bar{A}_h)$ and the $\varphi_i$ form a matching family for $\Phi$ for this cover. Since $\Phi$ is a sheaf, there is a unique amalgamation $\overline{\varphi}\in \Phi(\bar{A}_h)=\tilde\Phi([h])$.  This provides the required amalgamation of the original family. The fact that it is unique follows from the fact that any other amalgamation in $\tilde\Phi([h])$ would correspond to an amalgamation of the $\varphi_i$ as matching family in $\Phi$ and we have uniqueness there. We conclude that $\tilde\Phi$ is a sheaf.

Now let $\Psi$ be a sheaf on the Ehresmann site   $(\bG(\calC),T_J)$.
We want to show that $\check\Psi$ is a sheaf on the Grothendieck site $(\calC,J)$.
So let $\{\xymatrix@1{A_i\ar[r]^{f_i}&A}|\,i\in I\}\in J(A)$ and let
$\psi_i\in\check\Psi(A_i)=\Psi([1_{A_i}])$ for $i\in I$ be a matching family.
Then $\{\xymatrix@1{[f_i]\ar[r]|{\scriptscriptstyle\bullet}&[1_A]}\}\in T_J([1_A])$
and for each index $i\in I$ there is a horizontal arrow $\xymatrix@1@C=4em{[f_i]\ar[r]^{[f_i,1_{A_i}]}&[1_{A_i}]}$. Then let $\psi'_i=\Psi([f_i,1_{A_i}])(\psi_i)\in\Psi([f_i])$.
These form a matching family for the cover  $\{\xymatrix@1{[f_i]\ar[r]|{\scriptscriptstyle\bullet}&[1_A]}\}$ and since $\Psi$ is a sheaf, there is a unique amalgamation $\overline{\psi}\in \Psi([1_A])=\check\Psi(A)$.
This also an amalgamation for the original matching family $\psi_i$.
Uniqueness follows from the fact that amalgamations for the $\psi'_i$ in $\Psi$ correspond precisely to amalgamations for the $\psi_i$ in $\check\Psi$.

We conclude that the equivalence of presheaf categories $\mathbf{PreSh}(\calC)\simeq\mathbf{PreSh}(\bG(\calC))$ restricts to an equivalence of sheaf categories $\sh(\calC,J)\simeq\sh(\bG(\calC),T_J)$.
 \end{proof}

\subsection{Functors Between Categories of Sites}
\label{sec:functors-between-categories-of-sites}

Let $(\calC,J)$ and $(\calC',J')$ be two Grothendieck sites.
A functor $F\colon\calC\to\calC'$ induces a geometric morphism $\varphi_F\colon\sh(\calC',J')\to\sh(\calC,J)$ (with $(\varphi_F)^*$ given
by composition with $F$) if and only if $F$ is both covering preserving and covering flat \cite{johnstone2002, Karazeris2004}.
We recall the definition of these concepts.

\begin{dfn}
For Grothendieck sites $(\calC,J)$ and $(\calC',J')$ a functor $F\colon\calC\to\calC'$
is
\begin{enumerate}
\item
{\em covering preserving} if for any covering sieve $\calA\in J(A)$, its image is again a covering sieve; i.e., $F\calA\in J'(FA)$;
\item
{\em covering flat} if for each finite diagram $D\colon \calI\to\calC$ and any cone $T$ over $F\circ D$ in $\calC'$ with vertex $U$, the sieve
$$
\{h\colon V\to U\,|\, Th\mbox{ factors through the $F$-image of some cone over }D\}
$$
is a covering sieve in $\calC'$.
\end{enumerate}
Such a functor is called a {\em morphism between Grothendieck sites}.
\end{dfn}

\begin{rmk}
If the sites have all finite limits we could require the functors between them to just preserve those limits. However, our functors $\mathbf L$ and $\mathbf G$ don't preserve
this property (for instance, when $\calC$ is a site with all finite limits, ${\mathbf G}(\calC)$ does not necessarily have products in its vertical category), so in this case it makes more sense to work with covering-flat morphisms.
\end{rmk}

We want to use the results from Remark \ref{D:arrowcorrespondence} and
Proposition \ref{prop:sheaves-G-and-LG-and-C-and-GC} to introduce the corresponding concepts for double functors between ordered groupoids to characterize the morphisms between Ehresmann sites that give rise to geometric morphisms between the induced sheaf toposes. These will then be called  {\em morphisms of Ehresmann sites}.

We need the notion of a cone over a diagram in a double category. The relevant notion for ordered groupoids is as follows.

\begin{dfn}
\begin{enumerate}
\item  A {\em finite diagram} in a double category ${\calG}$ consists of a finite ordered groupoid ${\mathbb{I}}$ and a double functor $D\colon {\mathbb{I}}\to\calG$.
We write $D_i=D(i)$ for any object $i$ in ${\mathbb I}$, and $\xymatrix@1{D_i\ar[r]^{D_\alpha}&D_{i'}}$ for the image of a horizontal arrow
$\xymatrix@1{i\ar[r]^\alpha &i'}$ under $D$
and $\xymatrix@1{D_i\ar[r]^{D_{(i,i')}}|-{\scriptscriptstyle\bullet}&D_{i'}}$ for the image of a vertical arrow $\xymatrix@1{i\ar[r]|-{\scriptscriptstyle\bullet}&i'}$ under $D$.
\item An {\em hv-cone} over a diagram $D\colon {\mathbb I}\to\calG$ consists of an object $U$ and a family of arrows
$$
\xymatrix{U\ar[r]^{\xi_i}&E_i\ar[d]|-{\scriptscriptstyle\bullet}\\&D_i}
$$
for each $i\in I$ such that for each horizontal arrow $\xymatrix@1{i\ar[r]^\alpha&i'}$ in $\mathbb I$, the following triangle of horizontal arrows exists and commutes:
$$
\xymatrix{
&U\ar[dl]_{\xi_i}\ar[dr]^{\xi_{i'}}
\\
E_i\ar[rr]_{D_\alpha|_{E_i}}&&E_{i'}}
$$
and for each vertical arrow $\xymatrix@1{i\ar[r]|-{\scriptscriptstyle\bullet}&i'}$, $\xi_i=\xi_{i'}$.
\end{enumerate}
\end{dfn}

With this terminology in place we can define the notion of being covering flat for maps between Ehresmann sites as in the following definition.

\begin{dfn}
A {\em morphism of Ehresmann sites} $(\calG,T)\to(\calG',T')$ is a double functor
$\calG\to\calG'$ which satisfies the following two conditions:
\begin{itemize}
\item
It is {\em covering preserving}: If $\calA \in T(A)$ then $F\calA \in T'(FA)$.
\item
It is {\em covering flat}: For each finite diagram, $D\colon {\mathbb I}\to\calG$  and each hv-cone
$$
\xymatrix@R=1em{
U\ar[r]^{\xi_i}
&E_i \ar[d]|-{\scriptscriptstyle\bullet}\\ &FD_i}\mbox{ with }i\in I,
$$
over $FD$ in $\calG'$,
there is a covering sieve
$$
\left\{\xymatrix@1{U'_{k}\ar[r]|-{\scriptscriptstyle\bullet}&U}|k\in K\right\}\in T(U)
$$
such that for each $k\in K$ there is an hv-cone
$$\xymatrix@R=1em{T_k\ar[r]^{\theta_{ik}}&A_{ik}\ar[d]|-{\scriptscriptstyle\bullet}\\&D_i} \quad\mbox{ in }\calG$$
and a diagram
$$\xymatrix{U_k'\ar[r]^{h_k}&T_k'\ar[d]|-{\scriptscriptstyle\bullet}\\&FT_k}$$
such that
$$
\xymatrix@R=1em{
U_k'\ar[r]^{\xi_i|_{U_k'}}& E_i'\ar[d]|-{\scriptscriptstyle\bullet}\ar@{}[dr]|\equiv &
U_k'\ar[r]^{h_k}&T_k' \ar[r]^{F\theta_i|_{T_k'}} & E_i'\ar[d]|-{\scriptscriptstyle\bullet}
\\
&E_i\ar[d]|-{\scriptscriptstyle\bullet}&&&FA_{ik}\ar[d]|-{\scriptscriptstyle\bullet}
\\
&FD_i&&&FD_i
}
$$
for all objects $i$ in $\mathbb I$.
\end{itemize}
\end{dfn}

\begin{proposition}\label{cp}
Given left-cancellative Grothendieck sites $(\calC,J)$ and $(\calC',J')$, a functor $F\colon \calC\to\calC'$ is covering preserving if and only if its image $\mathbf{G}(F)$ is covering preserving as a morphism of Ehresmann sites
$(\mathbf{G}(\calC),T_J)\to(\mathbf{G}(\calC'),T_{J'})$.
\end{proposition}

\begin{proof}
Suppose first that $F$ is covering preserving and let
\(
\{
  \xymatrix{
  [m_i] \ar[r]|{\bullet} & [m]
  }
\}_{i\in I}
\)
be a vertical covering sieve of $[m\colon A\rightarrow B]$
in $\mathbf{G}(\mathcal{C}).$
This provides a covering sieve
$\{f_i\colon A_i\rightarrow A\}_{i\in I}$
of $A$ in $\mathcal{C}$ such that, for all $i\in I,$
$m_i = mf_i.$
Since $F$ is a covering-preserving functor,
the $F(f_i)$ cover $F(A)$  and
$F(m_i) = F(mf_i) = F(m)F(f_i)$
in $\mathcal{C'}$ for all $i\in I.$
Therefore, the data
\(
\{
  \xymatrix{
  [F(m_i)] \ar[r]|{\bullet} & [F(m)]
  }
\}_{i\in I}
\)
are a covering vertical sieve of $[F(m)]$ in $\mathbf{G}(\mathcal{C'});$
$\mathbf{G}(F)$ is covering preserving.

Conversely, suppose that $\mathbf{G}(F)$ is covering preserving and
let $\{m_i\colon A_i\rightarrow A\}_{i\in I}$ be a covering sieve of $A$ in
$\mathcal{C}.$
Then
\(
\{
  \xymatrix{
  [m_i] \ar[r]|{\bullet} & [1_A]
  }
\}_{i\in I}
\)
is a covering vertical sieve of $[1_A]$ in $\mathbf{\mathcal{C}}$ and,
since $\mathbf{G}(F)$ is covering preserving,
\(
\{
  \xymatrix{
  [F(m_i)] \ar[r]|{\bullet} & [F(1_A)]
  }
\}_{i\in I}
\)
is a covering vertical sieve of $[F(1_A)]$ in $\mathbf{G}(\mathcal{C}').$
Therefore, the $F(m_i)$ are a covering sieve of $F(A)$ in $\mathcal{C}';$
$F$ is covering preserving.
\end{proof}

\begin{proposition}\label{G-cf}
Given left-cancellative Grothendieck sites $(\calC,J)$ and $(\calC',J')$, if a functor $F\colon \calC\to\calC'$ is covering flat then its image $\mathbf{G}(F)$ is covering flat as a morphism of Ehresmann sites,
$(\mathbf{G}(\calC),T_J)\to(\mathbf{G}(\calC'),T_{J'})$.
\end{proposition}

\begin{proof}
Let $F\colon (\calC,J)\to (\calC',J')$ be covering flat.
We want to show that $\mathbf{G}(F)\colon (\mathbf{G}(\calC),T_J)\to(\mathbf{G}(\calC'), T_{J'})$ is covering flat as a map of Ehresmann sites.
So let $D\colon {\mathbb I}\to\mathbf{G}(\calC)$ be a finite diagram.
Note that in this case each $D_i$ has the form $[\xymatrix@1{D_i'\ar[r]^{d_i}&D_i}]$ where $\xymatrix@1{D_i'\ar[r]^{d_i}&D_i}$ is an arrow in $\calC$.
Now let
$$\xymatrix@C=5em@R=1.5em{
[U'\stackrel{u}{\longrightarrow}U]\ar[r]^{\xi_i}&[E_i'\stackrel{e_i}{\longrightarrow}E_i]\ar[d]|-{\scriptscriptstyle\bullet}
\\
&[FD_i'\stackrel{Fd_i}{\longrightarrow}FD_i]
}
$$
be an hv-cone over the diagram $\mathbf{G}(F)\circ D\colon{\mathbb I}\to \mathbf{G}(\calC')$.
This means that $E_i=FD_i$ and there is an arrow $E_i'\stackrel{v_i}{\longrightarrow} FD_i$ in $\calC'$ such that $Fd_i\circ v_i=e_i$ and such that $\xymatrix@1{U'\ar[r]^{v_i\xi_i}&FD_i}$ is a cone in $\calC'$ over the diagram $F\hat{D}\colon\mathbf{L}(\mathbb{I})\to\calC'$, where $\hat{D}$ is the adjunct of $D$.
In particular, if $D_i$ has the form $[D_i'\stackrel{d_i}{\longrightarrow}D_i]$ then $\hat{D}_i=D_i'$.

Since $F$ is covering flat, there is a Grothendieck covering $\{U'_k\stackrel{\varphi_k}{\longrightarrow}U'|\,k\in K\}\in J(U')$ such that $v_i\xi_i\varphi_k=F(\theta_{ik})\psi_k$ where $\psi_k\colon U_k'\rightarrow FT_k$ for some cone $T_k\stackrel{\theta_{ik}}{\longrightarrow}D_i'$ over $\hat{D}$ in $\calC$.

This gives rise to an Ehresmann covering $\{[\xymatrix@1{U_k'\ar[r]^{u\varphi_k}&U}]\xymatrix@1{\ar[r]|-{\scriptscriptstyle\bullet}&}[\xymatrix@1{U'\ar[r]^u&U}]|\,k\in K\}\in T_{J'}$ with the following diagrams in $\mathbf{G}(\calC')$:
$$
\xymatrix@C=4em{
[U_k'\stackrel{u\varphi_k}{\longrightarrow}U]\ar[r]^-{[1_{U_k'}]}&[U_k'\stackrel{\psi_k}{\longrightarrow}FT_k]\ar@{}[dr]|\le \ar[d]|-{\scriptscriptstyle\bullet}\ar[r]^-{[1_{U_k'}]} & [U_k'\stackrel{F(d_i\theta_{ik})\psi_k}{\longrightarrow}FD_i]\ar[d]|-{\scriptscriptstyle\bullet}
\\
&[FT_k\stackrel{1_{FT_k}}{\longrightarrow}FT_k]\ar[r]_-{[1_{FT_k}]}&[FT_k\stackrel{F(d_i\theta_{ik})}{\longrightarrow}FD_i]\ar[d]|-{\scriptscriptstyle\bullet}\\
&&[FD_i'\stackrel{Fd_i}{\longrightarrow}FD_i]
}
$$
and
$$
\xymatrix@C=4em{
[U_k'\stackrel{u\varphi_k}{\longrightarrow}U]\ar[d]|-{\scriptscriptstyle\bullet}\ar[r]^-{[1_{U_k'}]}\ar@{}[dr]|\le & [U_k'\stackrel{e_i\xi_i\varphi_k}{\longrightarrow}FD_i]\ar[d]|-{\scriptscriptstyle\bullet}
\\
[U'\stackrel{u}{\longrightarrow}U]\ar[r]_-{[\xi_i]}&[E_i'\stackrel{e_i}{\longrightarrow}FD_i]\ar[d]|-{\scriptscriptstyle\bullet}
\\
&[FD_i'\stackrel{Fd_i}{\longrightarrow}FD_i]
}
$$
Note that this means that $[\xi_i]|_{[U_k'\stackrel{u\varphi_k}{\longrightarrow}U]}=[1_{U_k'}]$ as in this last diagram. Furthermore, $e_i\xi_i\varphi_k=F(d_i)v_i\xi_i\varphi_k=F(d_i\theta_{ik})\psi_k$, so we have that the horizontal arrows on the tops of these diagrams are equal as required.

We conclude that $\mathbf{G}$ is covering flat.
\end{proof}

\begin{proposition}\label{L-cf}
Given Ehresmann sites $(\calG,T)$ and $(\calG',T')$, if the double functor $M\colon \calG\to\calG'$ is covering flat then its image $\mathbf{L}(M)$ is covering flat as a morphism of Grothendieck sites $(\mathbf{L}(\calG),J_T)\to (\mathbf{L}(\calG'),J_{T'})$.
\end{proposition}

\begin{proof}
Let $M\colon (\calG,T)\to(\calG',T')$ be covering flat.
Now let $D\colon\calJ\to \mathbf{L}(\calG)$ be a diagram, where $\calJ$ is a finite left-cancellative category.
By the adjunction, $\mathbf{G}\dashv\mathbf{L}$,  this induces a diagram
$\hat{D}\colon \mathbf{G}(\calJ)\to\calG$.
Now let $\xymatrix@1{U\ar[r]^-{\xi_j}&\mathbf{L}(M)(D_j)}$ be a cone over $\mathbf{L}(M)\circ D$ in $\mathbf{L}(\calG')$. Note that each $\xi_j$ has the form
$$\xymatrix{U\ar[r]^-{\hat{\xi}_j}&E_j\le \mathbf{L}(M)(D_j)=M(D_j).}$$
So the $\xi_j$ give rise to an hv-cone over $M\circ\hat{D}$ in $\calG'$ with components:
$$
\xymatrix{U\ar[r]^{\hat{\xi}_j}&E_j\ar[d]|-{\scriptscriptstyle\bullet}\\&M(D_j)}$$

Since $M$ is covering flat there is an Ehresmann covering
$\{\xymatrix@1{U_k'\ar[r]|-{\scriptscriptstyle\bullet}&U}|\,k\in K\}\in T'(U)$ such that for each $k\in K$ there is an hv-cone
$$\xymatrix{T_k\ar[r]^{\tau_k}&D_{j,k}'\ar[d]|-{\scriptscriptstyle\bullet}\\&D_j}$$
in $\calG$ with arrows $\xymatrix@1{U_k'\ar[r]^{\theta_k}&T_k'\ar[r]|-{\scriptscriptstyle\bullet}&MT_k}$ for each $k\in K$, such that
$$
\xymatrix{
U_k'\ar[r]^{\hat{\xi}_j|_{U_k'}}&E_{i,k}'\ar[d]|-{\scriptscriptstyle\bullet}\ar@{}[dr]|{\equiv} & U_k'\ar[r]^{\theta_k} & T_k'\ar[r]^{M\tau_k|_{T_k'}} & D_{j,k}''\ar[d]|-{\scriptscriptstyle\bullet}
\\
&E_j\ar[d]|-{\scriptscriptstyle\bullet}&&& MD_{j,k}'\ar[d]|-{\scriptscriptstyle\bullet}
\\
&M(D_j)&&&M(D_j)
}
$$
It follows that the family $\{U_k'\le U|\,k\in K\}$ is a Grothendieck cover of $U$ in $\mathbf{L}(\calG')$ and the $T_k$ for $k\in K$ give cones in $\mathbf{L}(\calG)$ such that the cones $((\hat{\xi})_j|_{U_k'},\le)$ factor through their images. We conclude that $\mathbf{L}(M)$ is covering flat.
\end{proof}

\begin{prop}\label{GL-cf}
For a double functor $M\colon (\calG, T)\to(\calG',T')$ between Ehresmann sites with maximal elements, if the double functor  $\mathbf{G}\mathbf{L}(M)\colon (\mathbf{G}\mathbf{L}(\calG),T_{J_T})\to (\mathbf{G}\mathbf {L}(\calG'),T_{J_{T'}})$ is covering flat, then $M$
is covering flat.
\end{prop}

\begin{proof}
Let $M\colon \calG\to\calG'$ be a double functor such that $\mathbf{G}\mathbf{L}(M)$ is covering flat.
	Let $D\colon\mathbb{I}\to \calG$ be a finite diagram, and
\begin{equation}\label{hv-cone}
\xymatrix{U\ar[r]^{\xi_i} & E_i\ar[d]|-{\scriptscriptstyle\bullet}\\&MD_i}
\end{equation} an hv-cone in $\calG'$.
For each object $D_i$, let $\hat{D}_i$ be the maximal element in the vertical category with $\xymatrix@1{D_i\ar[r]|-{\scriptscriptstyle\bullet}&\hat{D}_i}$.
Then define the diagram $\bar{D}\colon\mathbb{I}\to\mathbf{G}\mathbf{L}(\calG)$ by
$\bar{D}(i)=(D_i,\hat{D}_i)$, $\bar{D}(\xymatrix@1{i\ar[r]^-\alpha&i'})=(\xymatrix@1{(D_i,\hat{D}_i)\ar[r]^{D_\alpha}&(D_{i'},\hat{D}_{i'})})$, and $\bar{D}(\xymatrix@1{i\ar[r]|-{\scriptscriptstyle\bullet}&i'})=(\xymatrix@1{(D_i,\hat{D}_i)\ar[r]|-{\scriptscriptstyle\bullet}&(D_{i'},\hat{D}_{i'})=(D_{i'},\hat{D}_{i})})$. It follows that $\mathbf{G}\mathbf{L}(M)(\bar{D}_i)=(MD_i,M\hat{D}_i)$.
The hv-cone (\ref{hv-cone}) gives rise to an hv-cone
$$
\xymatrix{
(U,U)\ar[r]^{\xi_i}&(E_i,M\hat{D}_i)\ar[d]|-{\scriptscriptstyle\bullet}\\&(MD_i,M\hat{D}_i)\rlap{\,\,$=M\bar{D}_i$}
}
$$
in $\mathbf{G}\mathbf{L}(\calG')$.
Since $\mathbf{G}\mathbf{L}(M)$ is covering flat there is an Ehresmann covering
$\{\xymatrix@1{(U'_\ell,U)\ar[r]|-{\scriptscriptstyle\bullet}&(U,U)}|\,\ell\in L\}\in T_{J_{T'}}(U,U)$ with for each $\ell\in L$ an hv-cone
$$
\xymatrix{
(T_\ell,\hat{T}_\ell)\ar[r]^{\tau_{\ell,i}}&(A_{\ell,i},\hat{D}_{i})\ar[d]|-{\scriptscriptstyle\bullet}\\&(D_i,\hat{D}_i)
}
$$
in $\mathbf{G}\mathbf{L}(\calG)$ with a horizontal arrow $\xymatrix@1{(U_\ell',U)\ar[r]^{\theta_\ell}&(T_\ell',M\hat{T}_\ell)}$ such that
$$
\xymatrix@C=4em{
(U_\ell',U)\ar[r]^{\theta_\ell}&(T_\ell',M\hat{T}_\ell)\ar[r]^{(M\tau_{\ell,i})|_{T_\ell'}}&(A_{\ell,i}',M\hat{D}_{i})\ar[d]|-{\scriptscriptstyle\bullet}\ar@{}[dr]|\equiv &(U_\ell',U)\ar[r]^{\xi_i|_{U_\ell'}}&(E_{\ell,i}',M\hat{D}_i)\ar[d]|-{\scriptscriptstyle\bullet}
\\
&&(MA_{\ell,i},M\hat{D}_{i})\ar[d]|-{\scriptscriptstyle\bullet}&&(E_i,M\hat{D}_i)\ar[d]|-{\scriptscriptstyle\bullet}
\\
&&(MD_i,M\hat{D}_i)&&(MD_i,M\hat{D}_i)}
$$

From this data we obtain an Ehresmann covering $\{\xymatrix@1{U_\ell'\ar[r]|-{\scriptscriptstyle\bullet}&U}|\,\ell \in L\}\in T(U)$ with for each $\ell\in L$ an hv-cone,
$$\xymatrix{T_\ell\ar[r]^{\tau_{\ell,i}}&A_{\ell,i}\ar[d]|-{\scriptscriptstyle\bullet}\\&D_i}$$ and a horizontal arrow $\xymatrix@1{U_\ell'\ar[r]^{\theta_\ell}&T_\ell'}$ such that
$$
\xymatrix@C=4em{
U'_\ell\ar[r]^{\theta_\ell}&T_\ell'\ar[r]^{(M\tau_{\ell,i})|_{T_\ell'}}&A_{\ell,i}'\ar[d]|-{\scriptscriptstyle\bullet}\ar@{}[dr]|\equiv&U_\ell'\ar[r]^{\xi_i|_{U_\ell'}}&E_i'\ar[d]|-{\scriptscriptstyle\bullet}\\
&&MA_{\ell,i}\ar[d]|-{\scriptscriptstyle\bullet}&&E_i\ar[d]|-{\scriptscriptstyle\bullet}\\
&&MD_i&&MD_i
}
$$
as required. We conclude that $M$ is covering flat.

\end{proof}

\begin{prop}\label{LG-cf}
A functor $F\colon(\calC,J)\to(\calC',J')$ is covering flat and covering preserving if and only if its image $$\mathbf{L}\mathbf{G}(F)\colon (\mathbf{L}\mathbf{G}(\calC),J_{T_{J}})\to  (\mathbf{L}\mathbf{G}(\calC'),J_{T_{J'}})$$ is.
\end{prop}

\begin{proof}
This follows from the fact that the components of $\eta$ induce isomorphisms between the sheaf categories as shown by Lawson and Steinberg.
There is also a straightforward direct proof in terms of diagrams and cones.
\end{proof}

We derive from these propositions that the covering-flat covering-preserving double functors between Ehresmann sites are precisely the morphisms that give rise to geometric morphisms between the corresponding categories of sheaves.

\begin{thm}\label{geom}
A double functor $M\colon (\calG,T)\to(\calG',T')$ gives rise to a geometric morphism $\sh(\calG',T')\to\sh(\calG,T)$ if and only if it is covering preserving and covering flat.
\end{thm}

\begin{proof}
If $M$ is covering flat and covering preserving, then $\bL(M)$ is covering flat and covering preserving by Propositions \ref{cp} and \ref{L-cf} and this implies that
$\bL(M)^*$ is part of a geometric morphism and hence $M^*$ is as well by Remark \ref{D:arrowcorrespondence}.

Conversely, if $M^*$ is part of a geometric morphism, so is $\bL(M)^*$ by Remark \ref{D:arrowcorrespondence}.  So $\bL(M)^*$ is covering flat and covering preserving, and hence $M$ is covering preserving by Proposition \ref{cp} and $\bG\bL(M)$ is covering flat by Proposition \ref{G-cf}. But then $M$ is also covering flat by Proposition \ref{GL-cf}.
\end{proof}

We will now write {\bf lcGsite} for the 2-category of left-cancellative Grothendieck sites with covering-preserving covering-flat morphisms and {\bf Esite} for the 2-category of Ehresmann sites with covering-preserving covering-flat morphisms and ${\bf Esite}_{\mbox{\scriptsize max}}$ for the Ehresmann sites where each object is below a unique maximal object.
Then we conclude from the previous propositions that

\begin{thm}\label{adj-sites}
The functors $\mathbf{G}$ and $\mathbf{L}$ induce a 2-adjoint biequivalence
$$
\mbox{\bf lcGsite}\simeq\mbox{\bf Esite}_{\mbox{\scriptsize max}}.
$$
\end{thm}

\begin{rmk} Theorem \ref{geom} can now be seen as saying that the equivalence in Theorem \ref{adj-sites} is an equivalence of representations of \'etendues.
\end{rmk}

\subsection{The Comparison Lemma}
\label{sec:comparison-lemma}

To further investigate how morphisms of Ehresmann sites correspond to morphisms between \'etendues, we want to consider which morphisms would induce an equivalence between the corresponding \'etendues.

For Grothendieck sites, the Comparison Lemma \cite{kock-moerdijk-1991}
provides a list of sufficient conditions on a morphism
$F: (\mathcal{C}, J)\rightarrow (\mathcal{C}', J')$
to guarantee that the induced geometric morphism
$\sh(F): \sh(\mathcal{C}', J')\rightarrow \sh(\mathcal{C}, J)$
between the sheaf categories is an equivalence. The comparison lemma checks the following four conditions for maps between sites.

\begin{definition} \label{G-characteristics}
A morphism $F\colon (\mathcal{C}, J)\rightarrow (\mathcal{C}', J')$ of Grothendieck sites
is:
\begin{enumerate}[(GS.1)]
\item
\emph{locally full} if for each arrow $g\colon F(C)\rightarrow F(D)$ in
$\mathcal{C}',$ there exists a cover
$\left(\xi_i\colon C_i\rightarrow C\right)_{i\in I}$ in $\mathcal{C}$
with maps
$\left(f_i\colon C_i\rightarrow D\right)_{i\in I}$
such that $g\circ F(\xi_i) = F(f_i)$ for all $i\in I.$

\item
\emph{locally faithful} if for each pair of maps $f, f'\colon C\rightarrow D$ in $\mathcal{C}$
with $F(f) = F(f'),$ there exists a cover $(\xi_i)_{i\in I}$ of $C$ with
$f\circ\xi_i = f'\circ\xi_i$ for all $i\in I.$

\item
\emph{locally surjective on objects} if for each object $C'$ of
$\mathcal{C}',$ there exists a covering family of the form
$(F(C_i)\rightarrow C')_{i\in I}$ in $\calC'$.

\item
\emph{co-continuous} if for each cover $(\xi_i\colon C'_i\rightarrow F(C))_{i\in I}$ in $\mathcal{C}',$ the set of arrows $f\colon D\rightarrow C$
in $\mathcal{C},$ such that $F(f)$ factors through some $\xi_i,$ covers
$C$ in $\mathcal{C}.$
\end{enumerate}\end{definition}

Here is a slightly reformulated version of the Comparison Lemma (to take into account that we do not assume that our sites are closed under finite limits) as stated in \cite{kock-moerdijk-1991}.

\begin{theorem}[Comparison Lemma for Grothendieck Sites]
\label{thm:cl-grothendieck}
Let $F\colon (\mathcal{C}, J)\rightarrow (\mathcal{C}', J')$ be a morphism of Grothendieck sites.
If $F$ satisfies conditions (GS.1)-(GS.3), then the functor
$$F^*\colon \sh(\mathcal{C}', J')\rightarrow \sh(\mathcal{C}, J)$$
defined by composition with $F$
is full and faithful.
If $F$ further satisfies (GS.4), then $F^*$ is an equivalence. \qed
\end{theorem}

\begin{rmk}
The reader may wonder whether this lemma fully characterizes morphisms that
induce equivalences between the induced sheaf-topoi. The closest result in this direction is that for essentially small sites for Grothendieck topoi there is the characterization of the category of topoi being a category of left fractions for the category of sites with site morphisms with respect to the morphisms that satisfy the comparison lemma. Unfortunately, this result cannot be restricted to left-cancellative sites and \'etendues: although it is possible to represent each geometric morphism by a cospan of morphisms of sites it is not always possible to take the middle site to be left cancellative even if the other two are.
\end{rmk}

Our goal in this section is to give corresponding properties for morphisms
between Ehresmann sites and leverage Remark \ref{D:arrowcorrespondence} to
obtain a comparison lemma for Ehresmann sites.

\begin{definition}\label{E-characteristics}
A double functor $M\colon (\mathcal{G}, T)\rightarrow (\mathcal{G}', T')$
of Ehresmann sites is:
\begin{enumerate}[(ES.1)]
\item
\emph{locally full} if, for any diagram
$\xymatrix@1{M(A)\ar[r]^-g& B'\ar[r]|-{\scriptscriptstyle\bullet}& M(B)}$
in $\mathcal{G}'$,
there exists a covering vertical sieve
$\{\xymatrix@1@C-0.5pc{A_i\ar[r]|-{\scriptscriptstyle\bullet}& A}\}_{i\in I}\in T(A)$
and a family of horizontal arrows
$\{f_i: A_i\rightarrow B_i\}_{i\in I}$ in $\mathcal{G}$
such that
$g |_{M(A_i)} = M(f_i)$ for all $i\in I.$

\item
\emph{locally faithful} if, for any two horizontal arrows
$f\colon A\rightarrow B_f$
and
$g\colon A\rightarrow B_g$
with $\xymatrix@1{B_f\ar[r]|-{\scriptscriptstyle\bullet}&B}$ and $\xymatrix@1{B_g\ar[r]|-{\scriptscriptstyle\bullet}&B}$
in $\mathcal{G}$
and $M(f) = M(g),$
there exists a covering vertical sieve
$\{\xymatrix@1@C-0.5pc{A_i\ar[r]|-{\scriptscriptstyle\bullet}& A}\}_{i\in I}\in T(A)$
with
$f|_{A_i} = g|_{ A_i}$ for all $i\in I.$

\item
\emph{locally surjective on objects} if, for each
object $A'$ of $\mathcal{G}',$
there is a set
$\{\xymatrix@1@C-0.5pc{M(A_i) \ar[r]^-{f_i}  &A'_i \ar[r]|-\bullet & A'}\}_{i\in I}$
of horizontal arrows in $\mathcal{G}'$
such that
$\{\xymatrix@1@C-0.5pc{A'_i\ar[r]|-{\scriptscriptstyle\bullet}& A'}\}_{i\in I}\in T'(A')$
is a covering vertical sieve of $A'$ in $\mathcal{G}'.$

\item
\emph{co-continuous} if, for all covering vertical sieves
$\{\xymatrix@1@C-0.5pc{A'_i\ar[r]|-{\scriptscriptstyle\bullet}& M(A)}\}_{i\in I}\in T'(M(A))$
of $M(A)$ in $\mathcal{G}',$
the set
$\{ \xymatrix@1@C-0.5pc{A_j\ar[r]|-{\scriptscriptstyle\bullet}& A} \colon
    \xymatrix@1@C-0.5pc{M(A_j)\ar[r]|-{\scriptscriptstyle\bullet}& A'_i}
    \mbox{ for some } i\in I \}_{j\in J}$
is a covering vertical sieve of $A$ in $\mathcal{G}.$
\end{enumerate}\end{definition}

Given the criteria expressed in the language of Ehresmann sites, the desired Comparison Lemma for Ehresmann sites is given here.

\begin{theorem}[Comparison Lemma for Ehresmann Sites]
\label{thm:cl-ehresmann}
Let $M: (\mathcal{G}, T)\rightarrow (\mathcal{G}', T')$ be a morphism of
Ehresmann sites.
If $M$ satisfies (ES.1)--(ES.3), then the functor
$M^*: \sh(\mathcal{G}', T')\rightarrow \sh(\mathcal{G}, T)$
is full and faithful.
If further $M$ satisfies (ES.4), then $M^*$ is an equivalence. \qed
\end{theorem}

\begin{proof}
We note that once one writes down what it means for $M$ to satisfy (ES.1)--(ES.3) and for $\mathbf{L}(M)$ to satisfy (GS.1)--(GS.3),
one obtains exactly the same diagrams, slightly differently interpreted, and we need to prove exactly the same results on both sides. So $M$ satisfies each of the conditions (ES.1)--(ES.3) precisely when $\mathbf{L}(M)$ satisfies the corresponding condition in (GS.1)--(GS.3).

For (ES.4) and (GS.4) the correspondence is straightforward once one realizes that any covering in $J_{T'}(A)$ in $\mathbf{L}(\calG')$  is generated by a covering in $T'(A)$ in $\calG'$
and a family of arrows factors through the generating set precisely when it factors through the whole covering.
So $M$ satisfies (ES.4) precisely when $\mathbf{L}(M)$ satisfies (GS.4).

The result then follows from Theorem \ref{thm:cl-grothendieck}
and Remark \ref{D:arrowcorrespondence}.
\end{proof}

\end{document}